\documentclass[12pt]{amsart}
\usepackage{geometry}   
\usepackage[colorlinks,citecolor = red, linkcolor=blue,hyperindex]{hyperref}
\usepackage{euscript,eufrak,verbatim, mathrsfs, amscd}
\usepackage[psamsfonts]{amssymb}
\usepackage{bbm}
\usepackage{graphicx}
 \usepackage{float}
\usepackage{float, tikz}

\usepackage{cite}

 \usepackage[all, cmtip]{xy}

\usepackage{upref, xcolor, dsfont}
\usepackage{amsfonts,amsmath,amstext,amsbsy, amsopn,amsthm}
\usepackage{enumerate}

\usepackage{url}

\usepackage{mathtools}
\usepackage{bookmark}

 \usepackage{euscript}
\usepackage{helvet}         
\usepackage{courier}        
\usepackage{type1cm}        
\usepackage{multicol}        
\usepackage[bottom]{footmisc}

\newtheorem{theorem}{Theorem}[section]
\newtheorem*{theorem*}{Theorem B}
\newtheorem{lemma}[theorem]{Lemma}

\newtheorem{proposition}[theorem]{Proposition}
\newtheorem{corollary}[theorem]{Corollary}
\newtheorem*{definition*}{Definition}
\newtheorem*{remark*}{Remark}

\newtheorem*{observation*}{Observation}

\newtheorem*{assumption*}{Assumption}

\theoremstyle{definition}

\newtheorem{remark}[theorem]{Remark}

\begin{document}

\title[Toeplitz operators]{Toeplitz operators on weighted Bergman spaces induced by a class of radial weights}

\author
{Yongjiang Duan}
\address
{Yongjiang DUAN: School of
Mathematics and Statistics, Northeast Normal University, Changchun, Jilin 130024, P.R.China}
\email{duanyj086@nenu.edu.cn}

\author
{Kunyu Guo}
\address
{Kunyu GUO: School of Mathematical Sciences, Fudan University, Shanghai 200433, P.R.China}
\email{kyguo@fudan.edu.cn}

\author
{Siyu Wang}
\address
{Siyu WANG: School of
Mathematics and Statistics, Northeast Normal University, Changchun, Jilin 130024, P.R.China}
\email{wangsy696@nenu.edu.cn}

\author{Zipeng Wang}
\address{Zipeng WANG: College of Mathematics and Statistics, Chongqing University, Chongqing,
401331, P.R.China}
\email{zipengwang2012@gmail.com, zipengwang@cqu.edu.cn}

\begin{abstract}
Suppose that $\omega$ is a radial weight on the unit disk that satisfies both forward and reverse doubling conditions. Using Carleson measures and $T1$-type conditions, we obtain necessary and sufficient conditions of the positive Borel measure $\mu$ such that the Toeplitz operator $T_{\mu,\omega}:L^p_a(\omega)\to L_a^1(\omega)$ is bounded and compact for $0<p\leq 1$. In addition,
we obtain a bump condition for the bounded Toeplitz operators with $L^1(\omega)$ symbol on $L^1_a(\omega)$. This generalizes a result of Zhu in \cite{zhu1989}.
\end{abstract}


\maketitle

\section{Introduction}
For $0< p<+\infty$, let $L^p(\omega)$ denote the set of measurable functions on the unit disk $\mathbb{D}$ with
\[
\int_\mathbb{D}|f(z)|^p\omega(z)dA(z)<+\infty,
\]
where $\omega$ is a positive integrable function on $\mathbb{D}$ and $dA(z)=\frac{1}{\pi}dxdy$ is the normalized area measure on the complex plane $\mathbb{C}$. The function $\omega$ is called a \textit{weight} on $\mathbb{D}$.
Let $H(\mathbb{D})$ be the set of analytic functions on $\mathbb{D}$. We denote the \textit{weighted Bergman spaces} by
\[
L^{p}_a(\omega)=L^p(\omega)\cap H(\mathbb{D}).
\]
If $\omega\equiv 1$, then $L_a^p(\omega)$ is the unweighted Bergman space $L_a^p(\mathbb{D})$.

A weight $\omega$ is \textit{radial} if $\omega(z)=\omega(|z|)$. A radial weight $\omega$ is called a \textit{$\mathcal{D}$ weight} \cite{PR19,PR1907} if it satisfies the \textit{forward and reverse doubling condition}, i.e., there exist constants $C_1,C_2\in(1,\infty)$ and $C_3\in[1,\infty)$ such that
\begin{equation}\label{Definition of D}
C_1\widehat{\omega}\bigg(1-\frac{1-r}{C_2}\bigg)\leq \widehat{\omega}(r)\leq C_3\widehat{\omega}\bigg(\frac{1+r}{2}\bigg), \quad r\in [0,1),
\end{equation}
where $\widehat{\omega}(z)=\int_{|z|}^{1}\omega(r)dr$ for all $z\in\mathbb{D}$. The $\mathcal{D}$ weight naturally arises in operator theory on spaces of analytic
functions. The classical example of $\mathcal{D}$ weight is the standard weight $(\alpha+1)(1-|z|^2)^\alpha$ on $\mathbb{D}$ when $-1<\alpha<\infty$ \cite[Chapter 4]{zhu2007}. In some sense, the class of $\mathcal{D}$ weights is the largest possible class of weights for which most of the results concerning the standard weighted Bergman space $L_a^p(dA_\alpha)$ hold through with appropriate modifications, where
$dA_\alpha(z)=(\alpha+1)(1-|z|^2)^\alpha dA(z)$ for $-1<\alpha<\infty$. For instance, Pel\'{a}ez and R\"{a}tty\"{a} \cite[Theorem 3 and Theorem 5]{PR19} showed that the class of $\mathcal{D}$ weights is the largest class of radial weights $\omega$ such that the Bergman projection $P_\omega$ satisfies the $L^\infty$-$\text{BMO}$ estimate and the Littlewood-Paley formula holds in $L_a^p(\omega)$. The recent textbook of Pavlovi\'{c} \cite[Section 3.6.2, 3.6.3]{Pav} contains more details of weighted Bergman spaces induced by radial weights.

Note that $L_a^2(\omega)$ is a reproducing kernel Hilbert space when $\omega$ is a $\mathcal{D}$ weight (see \cite{PR19}).
Let $K_{z}^\omega$ be the reproducing kernel of $L_a^2(\omega)$ and $\mu$ be a finite positive Borel measure. We define the \textit{Toeplitz operator} $T_{\mu,\omega}$ through the integral formula
\begin{align}\label{E:def-toeplitz}
T_{\mu,\omega}(f)(z)=\int_{\mathbb{D}}f(\zeta)\overline{K_{z}^{\omega}(\zeta)}d\mu(\zeta).
\end{align}
The measure $\mu$ is called the symbol of the Toeplitz operator $T_{\mu,\omega}$.

One basic (but non-trivial) problem concerning the Toeplitz operator $T_{\mu,\omega}$ is to characterize its boundedness and compactness on the Bergman space $L^p_a(\omega)$. On the unweighted Bergman space, Luecking \cite{L1987} solved this problem for Toeplitz operators with positive symbols. Zhu \cite{Zhu1988} then extended Luecking's result to the standard weighted Bergman spaces on bounded symmetric domains of $\mathbb{C}^n$.
Agbor, B\'ekoll\'e and Tchoundja \cite{ABT-2011} studied bounded and compact Toeplitz operators on the
endpoint Bergman space $L^1_a(\mathbb{D})$. The results for the boundedness and compactness of Toeplitz operators on Bloch spaces can be found in \cite{WL-2010,WZZ-2006}.

In a series of works \cite{PR2,PR,PR3,PR1,PR00,PR19,PR1907,PR4,PRS}, Pel\'{a}ez and R\"{a}tty\"{a} provided a framework to investigate operator theory and harmonic analysis problems on a large class of weighted Bergman spaces with radial weights. Very recently, Pel\'{a}ez, R\"{a}tty\"{a} and Sierra \cite{PRS} described the boundedness and compactness of Toeplitz operators with positive symbols between weighted Bergman spaces $L^p_a(\omega)$ and $L_a^q(\omega)$ when the weight $\omega$ is regular and $1<p,q<+\infty$. Recall a continuous radial weight $\omega$ is \textit{regular} if $\psi_{\omega}(r)\asymp(1-r)$ for $0\leq r<1$, where $\psi_{\omega}(r)=\frac{\widehat{\omega}(r)}{\omega(r)}$ is the distortion function \cite{S}. Notice that a regular weight is a $\mathcal{D}$ weight. For basic properties of these classes of weights, one can consult \cite{PR,PR19} and the references therein.

In the first part of this paper, we present necessary and sufficient conditions for boundedness and compactness of Toeplitz operators acting
on weighted Bergman spaces $L_a^p(\omega)$ induced by $\mathcal{D}$ weights for $0<p\leq 1$.

For $0<p,q<\infty$, a positive Borel measure $\mu$ on $\mathbb{D}$ is said to be a \textit{$q$-Carleson measure for} $L_{a}^{p}(\omega)$ if $L_{a}^{p}(\omega)$ can be boundedly embedded into $L^q(\mu)$, that is, there exists a constant $C$ such that
\[
\|f\|_{L^q(\mu)}\leq C\|f\|_{L^p(\omega)}, \quad \forall f\in L_a^p(\omega).
\]
Also a positive Borel measure $\mu$ is a \textit{vanishing~$q$-Carleson~measure~for} $L_{a}^{p}(\omega)$ if the identity operator $Id:L_{a}^{p}(\omega)\rightarrow L^{q}(\mu)$ is compact. One can consult \cite{PZ1,PZ2,PR,zhu2007} for more examples of (vanishing) Carleson measures on analytic function spaces.

In the case when $1<p<\infty$, the boundedness (compactness) of Toeplitz operators with positive symbols can be characterized exactly by Carleson measure (vanishing Carleson measure) type conditions. However, the Carleson measure type condition itself is not enough to deal with the case when $0<p\leq1$. Indeed, finding the extra condition is the primary difficulty in our study. We find an appropriate $T1$-type condition to achieve this goal when we seek suitable testing functions. Since one can not write down an explicit expression of the reproducing kernel $K_z^\omega$ for a general $\mathcal{D}$ weight, this difficulty makes us to encounter some obstacles while looking for suitable testing functions and desired estimates. By a delicate analysis on some classes of radial weights, we construct and study several auxiliary weights (see Section 2 for details). This leads us to deal with Toeplitz operators $T_{\mu,\alpha}$ instead of $T_{\mu,\omega}$ in the $T1$-type condition, where
\[
T_{\mu,\alpha}(f)(z)=
\int_{\mathbb{D}}\frac{f(\zeta)}{(1-z\bar{\zeta})^{\alpha+2}}d\mu(\zeta).
\]
One can refer to \cite[Chapter 4 and Chapter 7]{zhu2007} for more details about the Toeplitz operators $T_{\mu,\alpha}$ and the reproducing kernels of standard weighted Bergman spaces $L_{a}^{2}(dA_{\alpha})$.

Let $\omega$ be a weight and $t\in\mathbb{R}$. For $f\in H(\mathbb{D})$, we define
\[
\|f\|_{\mathcal{LB}^{t}_{\omega}}=|f(0)|+
\sup_{z\in\mathbb{D}}\frac{(1-|z|^{2})^{t+1}}{\omega(z)}\log\bigg(\frac{2}{1-|z|^{2}}\bigg)|f'(z)|.
\]

The following is our main result of this paper, which characterizes the boundedness of the Toeplitz operators on $L_{a}^{p}(\omega)$ for $p\in(0,1]$ and $\omega\in\mathcal{D}$.

\begin{theorem} \label{T:regular}
Suppose that $\omega$ is a $\mathcal{D}$ weight and $\mu$ is a positive Borel measure on $\mathbb{D}$. For $0<p\leq1$, set $\sigma^p(z)=\widehat{\omega}(z)$.  Then the Toeplitz operator $T_{\mu,\omega}: L^{p}_{a}(\omega)\to L^{1}_{a}(\omega)$ is bounded if and only if $\mu$ is a $\frac{1}{p}$-Carleson measure for $L^{1}_{a}(\omega)$ and there exists a positive constant $t_{0}=t_{0}(\omega)$ such that $\|T_{\mu,t}(1)\|_{\mathcal{LB}^{s}_{\sigma}}<\infty$, for each $t\geq t_{0}$ and $s=t-\frac{1}{p}+2$.
\end{theorem}

We obtain the little-o condition for the compactness of $T_{\mu,\omega}: L^{p}_{a}(\omega)\to L^{1}_{a}(\omega)$ as follows.

\begin{theorem}\label{RC}
Suppose that $\omega$ is a $\mathcal{D}$ weight and $\mu$ is a positive Borel measure on $\mathbb{D}$. For $0<p\leq1$, set $\sigma^p(z)=\widehat{\omega}(z)$.  Then the Toeplitz operator $T_{\mu,\omega}: L^{p}_{a}(\omega)\to L^{1}_{a}(\omega)$ is compact if and only if $\mu$ is a vanishing $\frac{1}{p}$-Carleson measure for $L^{1}_{a}(\omega)$ and there exists a positive constant $t_{0}=t_{0}(\omega)$ such that
\[
\lim_{|z|\to 1^-}\frac{(1-|z|^{2})^{s+1}}{\sigma(z)}\log\bigg(\frac{2}{1-|z|^{2}}\bigg)|(T_{\mu,t}1)'(z)|=0,
\]
for each $t\geq t_{0}$ and $s=t-\frac{1}{p}+2$.
\end{theorem}

\begin{remark} Note that the $T1$-type condition in Theorem \ref{T:regular} is not superfluous. For instance, let $d\mu=\omega dA$. Then $T_{\mu,\omega}$ is the Bergman projection $P_{\omega}$ on $L^2(\omega)$. In this case, $\mu$ is a 1-Carleson measure for $L_{a}^{1}(\omega)$. However, by \cite[Theorem 4]{PR1}, the projection $P_{\omega}$ from $L^{1}(\omega)$ to $L_{a}^{1}(\omega)$ is not bounded in general. This shows that the Carleson measure type condition itself is not enough to guarantee the boundedness of Toeplitz operators on $L_{a}^{1}(\omega)$.

\end{remark}

The full characterization of bounded Toeplitz operators with general symbols on $L^p_a(\omega)$ remains open (see \cite{TV2010}, \cite{YZ2019} and the references therein). Recall for $\varphi\in L^{1}(\omega)$, the \textit{Toeplitz operator} $T_{\varphi,\omega}$ defined on polynomials is
\[
T_{\varphi,\omega}(f)(z)=\int_{\mathbb{D}}f(\zeta)\overline{K_{z}^{\omega}(\zeta)}
\varphi(\zeta)\omega(\zeta)dA(\zeta).
\]
Zhu obtained a sufficient condition of bounded Toeplitz operators on $L^1_a(\mathbb{D})$ in \cite{zhu1989}. In the second part of this paper, we generalize Zhu's result to weighted Bergman spaces through the following ``bump condition".

Let $\omega(E)=\int_{E}\omega(z)dA(z)$, for any measurable set $E\subseteq\mathbb{D}$.
For $z\in\mathbb{D}$ and $r\in(0,1)$, denote by $B(z,r)=\{\zeta\in\mathbb{D}:|z-\zeta|<r\}$ the Euclidean disk and $D(z,r)=\{\zeta\in\mathbb{D}:\big|\frac{z-\zeta}{1-\bar{z}\zeta}\big|<r\}$ the pseudohyperbolic disk.

\begin{theorem} \label{T:bound10} Let $\omega$ be a $\mathcal{D}$ weight and $\varphi\in L^{1}(\omega)$. If there exist $\varepsilon>0$ and $r\in(0,1)$ such that
\[
\sup_{z\in\mathbb{D}}\frac{1}{\omega(D(z,r))^{1+\varepsilon}}
\int_{D(z,r)}|\varphi(\zeta)|\omega(\zeta)dA(\zeta)<+\infty,
\]
then the Toeplitz operator $T_{\varphi,\omega}$ is bounded on $L_{a}^{1}(\omega)$.
\end{theorem}

In this paper, we will use $a\lesssim b$ to represent that there exists a constant $C=C(\cdot)>0$ satisfying $a\leq Cb$, where the constant $C(\cdot)$ depends on the parameters indicated in the parenthesis, varying under different circumstances. Moreover, if $a\lesssim b$ and $b\gtrsim a$, then we write $a\asymp b$.

\section{Boundedness of Toeplitz operators $T_{\mu,\omega}$ on $L_{a}^{p}(\omega)$}
Recall the \textit{Bloch space} $\mathcal{B}$ is the Banach space of analytic functions $f$ on $\mathbb{D}$ such that
\[
\|f\|_{\mathcal{B}}:=|f(0)|+\sup_{z\in\mathbb{D}}(1-|z|^2)|f'(z)|<+\infty.
\]
It is well known that the dual space of unweighed Bergman space $L_a^1(\mathbb{D})$ can be identified with the Bloch space. This result was generalized to the weighted Bergman spaces induced by $\mathcal{D}$ weights in \cite{PR19}.
\begin{lemma}\label{L:bloch}\rm\cite[Theorem 3]{PR19}
Let $\omega$ be a $\mathcal{D}$ weight. Then $(L_a^1(\omega))^{*}\simeq \mathcal{B}$ via the $L_a^2(\omega)$-pairing
\[
\langle f,g\rangle_{L^2(\omega)}=\lim_{r\to 1}\int_{\mathbb{D}}f_r(z)\overline{g(z)}\omega(z)dA(z),
\]
where $f_r(z)=f(rz)$ for $r\in (0,1)$.
\end{lemma}

From \cite[Theorem 3]{PR19}, we also know that $\mathcal{D}$ weights is the largest class of radial weights to make the dual relation $(L_a^1(\omega))^{*}\simeq \mathcal{B}$ holds.

We say a radial weight $\omega\in\widehat{\mathcal{D}}$ \cite{PR00} if there is a constant $C$ such that
\[
\widehat{\omega}(r)\leq C\widehat{\omega}\bigg(\frac{1+r}{2}\bigg),\quad \forall r\in[0,1).
\]

Next, we recall some useful estimates of reproducing kernels associated with $\widehat{\mathcal{D}}$ weights
(see \cite[Theorem C, Lemma 6 and Lemma 8]{PRS}) and geometrical characterizations for Carleson measures of the weighted Bergman spaces induced by $\mathcal{D}$ weights (see \cite[Theorem 3.3]{PR00}).
\begin{proposition}\label{lla45} Let $\omega\in\widehat{\mathcal{D}}$ and $0<p<\infty$.
Then
\[
\|K_{z}^{\omega}\|_{L_{a}^{p}(\omega)}^{p}\asymp\int_{0}^{|z|}\frac{1}{\widehat{\omega}(t)^{p-1}(1-t)^{p}}dt,\quad |z|\to1^{-}
\]
and
\[
\|K_{z}^{\omega}\|_{\mathcal{B}}\asymp\frac{1}{\widehat{\omega}(z)(1-|z|)},\quad z\in\mathbb{D}.
\]
Moreover, there exists $r\in(0,1)$ (r only depends on $\omega$) such that
\[
|K_{z}^{\omega}(\zeta)|\asymp K_{z}^{\omega}(z),\quad \textmd{for~all}~z\in\mathbb{D}~\textmd{and}~\zeta\in D(z,r).
\]
\end{proposition}

\vskip 0.2in

\begin{lemma}\label{P:D} Let $\omega$ be a $\mathcal{D}$ weight and $0<p\leq q<\infty$. Then a positive Borel measure $\mu$ is a $q$-Carleson measure for $L_{a}^{p}(\omega)$ if and only if
\[
\sup_{z\in\mathbb{D}}\frac{\mu(D(z,r))}{\widehat{\omega}(z)^{\frac{q}{p}}(1-|z|)^{\frac{q}{p}}}<\infty, \quad \textmd{for~some}~r\in(0,1).
\]
\end{lemma}

By the above $L^1_{a}$-$\mathcal{B}$ dual relations, estimates of reproducing kernels and geometrical conditions of Carleson measures, we obtain a part of Theorem \ref{T:regular}.
\begin{proposition}\label{P:regular} Let $\mu$ be a positive Borel measure, $\omega$ be a $\mathcal{D}$ weight and $0<p\leq1$. If
$T_{\mu,\omega}:L_{a}^{p}(\omega)\rightarrow L_{a}^{1}(\omega)$
is bounded, then $\mu$ is a $\frac{1}{p}$-Carleson measure for $L_{a}^{1}(\omega)$.
\end{proposition}
\begin{proof}
Choose a real number $\alpha$ with $\alpha p>1$. Let $\beta=\alpha+1-\frac{1}{p}$. For $z\in\mathbb{D}$, we define
\[
f_{z}(\zeta)=\frac{(K_{z}^{\omega}(\zeta))^{\alpha}}
{\|K_{z}^{\omega}\|^{\beta}_{L^{\beta}(\omega)}}~\text{and}~g_{z}(\zeta)=\frac{(K_{z}^{\omega}(\zeta))^{\alpha}}
{\|K_{z}^{\omega}\|^{\alpha+1}_{L^{\alpha+1}(\omega)}}, \quad \zeta\in\mathbb{D}.
\]
By Proposition \ref{lla45}, we have
\begin{equation}\label{rp1}
\sup_{z\in\mathbb{D}}\|f_z\|_{L^{p}(\omega)}<\infty
\end{equation}
and
\begin{equation}\label{rpnn}
\sup_{z\in\mathbb{D}}\|g_z\|_{\mathcal{B}}<\infty.
\end{equation}
It follows from Lemma \ref{L:bloch} and the boundedness of $T_{\mu,\omega}$ that
\begin{equation*}\label{rp0}
\sup_{z\in\mathbb{D}}|\langle T_{\mu,\omega}(f_z),g_z\rangle_{L^{2}(\omega)}|<\infty.
\end{equation*}
From Proposition \ref{lla45}, we know that there exist positive constants $C_1,C_2,C_3$ and $r\in(0,1)$ such that for any $z\in\mathbb{D}$,
\begin{align}\label{rp3}
|\langle T_{\mu,\omega}(f_z),g_z\rangle_{L^{2}(\omega)}|&=|\langle f_{z},g_{z}\rangle_{L^{2}(\mu)}|\nonumber\\
&\geq C_1\widehat{\omega}(z)^{2\alpha-\frac{1}{p}}(1-|z|)^{2\alpha-\frac{1}{p}}
\int_{D(z,r)}|K_{z}^{\omega}(\zeta)|^{2\alpha}d\mu(\zeta)\nonumber\\
&\geq C_2\widehat{\omega}(z)^{2\alpha-\frac{1}{p}}(1-|z|)^{2\alpha-\frac{1}{p}}
\mu(D(z,r))(K_{z}^{\omega}(z))^{2\alpha}\nonumber\\
&\geq C_3
\frac{\mu(D(z,r))}{\widehat{\omega}(z)^{\frac{1}{p}}(1-|z|)^{\frac{1}{p}}}.
\end{align}
Hence, we obtain
\[
\sup_{z\in\mathbb{D}}\frac{\mu(D(z,r))}{\widehat{\omega}(z)^{\frac{1}{p}}(1-|z|)^{\frac{1}{p}}}<\infty.
\]
By Lemma \ref{P:D}, we conclude that $\mu$ is a $\frac{1}{p}$-Carleson measure for $L_{a}^{1}(\omega)$ and complete the proof.
\end{proof}

Now we are devoted to proving the $T1$-type condition of Theorem \ref{T:regular}. Before that, we are going to make some preparations. In order to construct some auxiliary weights mentioned in the introduction, we provide some discussions about $\widehat{\mathcal{D}}$, $\mathcal{D}$ and regular weights.

\begin{proposition}\label{property hat}
Let $0<s<+\infty$ and $t>-1$.
If $\omega\in\widehat{\mathcal{D}}$, then $\tau(z)=\widehat{\omega}(z)^{s}(1-|z|)^{t}$ belongs to $\widehat{\mathcal{D}}$.
\end{proposition}
\begin{proof}
For any $r\in[0,1)$, it holds that
\begin{align*}
\int_{r}^{1}\widehat{\omega}(\rho)^{s}(1-\rho)^{t}d\rho
&\leq\widehat{\omega}(r)^{s}\int_{r}^{1}(1-\rho)^{t}d\rho\lesssim\widehat{\omega}(\frac{1+r}{2})^{s}(1-r)^{t+1}\\
&\lesssim\widehat{\omega}(\frac{3+r}{4})^{s}\int_{\frac{1+r}{2}}
^{\frac{3+r}{4}}(1-\rho)^{t}d\rho
\leq\int_{\frac{1+r}{2}}^{1}\widehat{\omega}(\rho)^{s}(1-\rho)^{t}d\rho,
\end{align*}
where the second and the third inequality are both due to the forward doubling condition of $\omega$.
Thus, we obtain that $\tau\in\widehat{\mathcal{D}}$.
\end{proof}

The proof of the proposition above is provided by one of the anonymous referees, which improves the original statement and simplifies the proof.

\begin{corollary}\label{choose of cp}
Let $0<p<1$ and $\omega\in\widehat{\mathcal{D}}$. Then for each $c_{p}\in(2-\frac{1}{p},+\infty)$, the weight
$
\widehat{\omega}(z)^{\frac{1}{p}}(1-|z|)^{c_p+\frac{1}{p}-3}
$
belongs to $\widehat{\mathcal{D}}$.
\end{corollary}

\begin{remark}\label{rema-doubling}
If $\omega$ and $\nu$ belong to $\widehat{\mathcal{D}}$, then there exist constants $C_{1}>1$ and $C_{2}>1$ such that
\begin{align*}
\int_{r}^{1}\omega(s)ds \leq C_{1}\int_{\frac{1+r}{2}}^{1}\omega(s)ds
\text{ and }
\int_{r}^{1}\nu(s)ds \leq C_{2}\int_{\frac{1+r}{2}}^{1}\nu(s)ds,\quad \forall r\in[0,1).
\end{align*}
Let $p,q\in(0,\infty)$ and
\[
t_0>\textmd{max}\bigg\{p^{-1}\log_{2}C_{1}+p^{-1}-2,~\log_{2}C_{2}-2q\bigg\}.
\]
By \cite[Lemma 2.1]{PR00}, for each $t\geq t_{0}$, we have the following useful estimates,
\[
\int_{\mathbb{D}}\frac{\omega(z)}{|1-\bar{z}\zeta|^{pt+2p}}dA(z) \asymp\frac{\widehat{\omega}(\zeta)}{(1-|\zeta|)^{pt+2p-1}},\quad \zeta\in\mathbb{D}
\]
and
\[
\int_{\mathbb{D}}\frac{\nu(z)}{|1-\bar{z}\zeta|^{2q+t+1}}dA(z) \asymp\frac{\widehat{\nu}(\zeta)}{(1-|\zeta|)^{2q+t}}, \quad \zeta\in\mathbb{D}.
\]
\end{remark}
These two formulas above can be regarded as a weighted version of the classical integral estimates (see \cite[Lemma 3.10]{zhu2007}) and will be frequently used in the proof of Theorem \ref{T:regular} and Theorem \ref{RC}.
\vskip 0.2in

The next result (see \cite[Proposition 5]{PRS} and its proof) will play an important role in the subsequent work. It shows that there exists a regular weight $\kappa$ related to any given $\mathcal{D}$ weight $\omega$. Usually discussions on regular weights turn out to be more convenient to deal with than on $\mathcal{D}$ weights.
\begin{lemma}\label{prop-reverse}
Let $0<p<\infty$. If $\omega$ is a $\mathcal{D}$ weight, then $\kappa(z)=(1-|z|)^{-1}\widehat{\omega}(z)$ is regular and $\|f\|_{L^{p}(\omega)}\asymp\|f\|_{L^{p}(\kappa)}$, for each analytic function $f$ on $\mathbb{D}$.
\end{lemma}

From Lemma \ref{prop-reverse}, we get the two following useful results.

\begin{corollary}\label{coro-Lemma}
Let $\omega$ be a $\mathcal{D}$ weight and $\kappa(z)=(1-|z|)^{-1}\widehat{\omega}(z)$, $z\in\mathbb{D}$. Then there exist constants $C_{1}\in(0,1)$ and $C_{2}>1$ such that for any constant $C_{3}\in(0,C_{1}]$,
\[
\kappa(r)\geq C_{2}^{-1}\widehat{\kappa}(0)(1-r)^{\frac{1}{C_{3}}-1}, \quad r\in[0,1).
\]
\end{corollary}
\begin{proof} By Lemma \ref{prop-reverse}, $\kappa$ is regular. Then there exist constants $C_{1}\in(0,1)$ and $C_{2}>1$ depending only on $\omega$ such that
\begin{equation*}\label{PRO3}
C_{1}\kappa(r)(1-r)\leq\widehat{\kappa}(r)\leq C_{2}\kappa(r)(1-r), \quad r\in[0,1).
\end{equation*}
For any $C_{3}\in (0,C_1]$,
\[
h_{C_{3}}(r):=\frac{\widehat{\kappa}(r)}{(1-r)^{\frac{1}{C_{3}}}}
\]
is increasing in $r\in[0,1)$. Therefore,
\[ h_{C_{3}}(r)\geq h_{C_{3}}(0)=\widehat{\kappa}(0) \text{ and } \widehat{\kappa}(r)\geq\widehat{\kappa}(0)(1-r)^{\frac{1}{C_{3}}},
\quad r\in [0,1).
\]
Hence
\[
\kappa(r)\geq C_{2}^{-1}\widehat{\kappa}(0)(1-r)^{\frac{1}{C_{3}}-1}, \quad r\in[0,1).
\]
\end{proof}

\begin{corollary}\label{choose of c1}
If $\omega$ is a $\mathcal{D}$ weight, then there exists a constant $c_{1}\in(0,1)$ ($c_{1}$ only depends on $\omega$) such that
$
\widehat{\omega}(z)(1-|z|)^{c_1-2}
$
is regular.
\end{corollary}
\begin{proof} For any $z\in\mathbb{D}$, let
\[
\kappa(z)=\frac{\widehat{\omega}(z)}{1-|z|}.
\]
It follows from Lemma \ref{prop-reverse} that $\kappa$ is regular, and hence, there exists a constant $C>1$ depending only on $\omega$ such that
\begin{equation}\label{add2}
C^{-1}\kappa(r)(1-r)\leq \widehat{\kappa}(r)\leq C\kappa(r)(1-r), \quad r\in[0,1).
\end{equation}
Given a positive constant $c_{1}$. By an observation from \cite[pp. 10]{PR}, we know that $\frac{\kappa(z)}{(1-|z|)^{1-c_{1}}}$ is regular when $c_{1}>1-C^{-1}$. Here we give the detailed proof for the convenience of readers. By the decreasing property of the function $h_{C}(r):=\frac{\widehat{\kappa}(r)}{(1-r)^{\frac{1}{C}}}$ and the increasing property of the function $h_{C^{-1}}(r):=\frac{\widehat{\kappa}(r)}{(1-r)^{C}}$ over $[0,1)$ respectively, we deduce
\[
(1-r)^{C}\lesssim\widehat{\kappa}(r)\lesssim (1-r)^{\frac{1}{C}},\quad r\in[0,1).
\]
Together with an integration by parts, for $c_{1}>1-C^{-1}$, we obtain
\begin{equation}\label{add3}
\int_{r}^{1}\frac{\kappa(s)}{(1-s)^{1-c_{1}}}ds=\frac{\widehat{\kappa}(r)}{(1-r)^{1-c_{1}}}
+(1-c_{1})\int_{r}^{1}\frac{\widehat{\kappa}(s)}{(1-s)^{2-c_{1}}}ds.
\end{equation}
Combining \eqref{add2} with \eqref{add3}, we get the desired conclusion.
\end{proof}
In fact, for $\omega\in\widehat{\mathcal{D}}$, $\widehat{\omega}(z)(1-|z|)^{c_1-2}$ is regular for any $c_{1}\in(1,\infty)$
(one can prove this statement by a similar argument as in \cite[Lemma 1.7]{PR}). From Corollary \ref{choose of c1}, we know that
for $\omega\in\mathcal{D}$, $\widehat{\omega}(z)(1-|z|)^{c_1-2}$ is also regular for any $c_{1}\in(1-C^{-1},1]$. However, it fails for $\omega\in\widehat{\mathcal{D}}\backslash\mathcal{D}$.

\vskip 0.1in

Take a positive constant $c_{p}\in(2-\frac{1}{p},1)$ in Corollary \ref{choose of cp} when $p\in(0,1)$
and let $c_1$ be the constant in Corollary \ref{choose of c1}. Let $c=c_{p}$ when $p\in(0,1)$ and $c=c_{1}$ when $p=1$.
Then for $p\in(0,1]$, we define
\begin{equation}\label{E:nu}
\nu(z)=\widehat{\omega}(z)^{\frac{1}{p}}(1-|z|)^{c+\frac{1}{p}-3}.
\end{equation}
In conclusion, in Corollary \ref{choose of cp} and Corollary \ref{choose of c1}, we construct an auxiliary weight which belongs to $\widehat{\mathcal{D}}$ when $p\in(0,1]$, and especially it is regular when $p=1$. This auxiliary weight is crucial in the remaining proof of Theorem \ref{T:regular}
and Theorem \ref{RC}.
\vskip 0.1in

The next result gives a nice and useful property for regular weights.
\begin{proposition}\label{la45} \rm\cite[pp. 8]{PR} Let $\omega$ be a regular weight.
Then for any $z\in\mathbb{D}$ and $r\in (0,1)$,
\[
\omega(z)\asymp\omega(u),\quad\textmd{where}~u\in D(z,r).
\]
\end{proposition}

Before stating the next proposition, recall a sequence $\{z_{j}\}_{j=1}^{\infty}\subseteq\mathbb{D}$ is called \textit{$\delta$-separated} ($\delta>0$) if $\big|\frac{z_{j}-z_{k}}{1-\bar{z_{j}}z_{k}}\big|\geq\delta$ when $j\neq k$. For $r\in(0,1)$, we call a sequence $\{z_{j}\}_{j=1}^{\infty}$ is an \textit{$r$-lattice} if $\{z_{j}\}_{j=1}^{\infty}$ satisfies $\mathbb{D}=\bigcup\limits_{j=1}^{\infty}D(z_{j},5r)$ and $\frac{r}{5}$-separated. Note that every $z\in\mathbb{D}$ belongs to at most $N=N(r)$ pseudohyperbolic disks $D(z_{j},r)$, where $\{z_{j}\}_{j=1}^{\infty}$ is a separated sequence (see Lemma 12 of Chapter 2 in \cite{D1} or Lemma 3 in \cite{Lu}).

\vskip 0.1in
Now we are ready to prove the $T1$-type condition in the necessity part of Theorem \ref{T:regular}.
\begin{proposition}\label{P1:regular}
Suppose that $\omega$ is a $\mathcal{D}$ weight and $\mu$ is a positive Borel measure on $\mathbb{D}$. For $0<p\leq1$, set $\sigma^p(z)=\widehat{\omega}(z)$. If the Toeplitz operator $T_{\mu,\omega}: L^{p}_{a}(\omega)\to L^{1}_{a}(\omega)$ is bounded, then there exists a positive constant $t_{0}=t_{0}(\omega)$ such that  $\|T_{\mu,t}(1)\|_{\mathcal{LB}^{s}_{\sigma}}<\infty$, for each $t\geq t_{0}$ and $s=t-\frac{1}{p}+2$.
\end{proposition}
\begin{proof}
Let
$
\kappa(z)=(1-|z|)^{-1}\widehat{\omega}(z),~z\in\mathbb{D}.
$
From Corollary \ref{coro-Lemma}, there exist constants $C_{1}\in(0,1)$ and $C_{2}>1$ such that for any constant $C_{3}\in(0,C_{1}]$,
\begin{equation}\label{kappa}
\kappa(r)\geq C_{2}^{-1}\widehat{\kappa}(0)(1-r)^{\frac{1}{C_{3}}-1}, \quad r\in[0,1).
\end{equation}
Let $\nu$ be the weight given by \eqref{E:nu}.
It follows from Corollary \ref{choose of cp}, Remark \ref{rema-doubling} and Corollary \ref{choose of c1} that we can choose a constant
$C_{4}\in (0, C_1]$ such that
for each
\[
t\geq t_0:=\frac{C_{4}+1}{pC_{4}}-2>0,
\]
we have
\begin{equation}\label{PR02}
\int_{\mathbb{D}}\frac{\omega(z)}{|1-\bar{z}\zeta|^{pt+2p}}dA(z) \asymp\frac{\widehat{\omega}(\zeta)}{(1-|\zeta|)^{pt+2p-1}}, \quad \zeta\in\mathbb{D}
\end{equation}
and
\begin{equation}\label{PR03}
\int_{\mathbb{D}}\frac{\nu(z)}{|1-\bar{z}\zeta|^{2c+t+1}}dA(z) \asymp\frac{\widehat{\nu}(\zeta)}{(1-|\zeta|)^{2c+t}}, \quad \zeta\in\mathbb{D}.
\end{equation}

Note $s=t-\frac{1}{p}+2$ and $\sigma^p(z)=\widehat{\omega}(z)$. For $z\in\mathbb{D}$,
let
\[
h_{z}(\zeta)=\frac{(1-|z|^{2})^{s+1}}{\sigma(z)}
\frac{1}{(1-\zeta\bar{z})^{t+3}}, \quad \zeta\in\mathbb{D}.
\]
For $0<p\leq 1$, we have
\begin{equation}\label{E:testh}
\|h_{z}\|_{L_{a}^{p}(\omega)}=\frac{(1-|z|^{2})^{s+1}}{\sigma(z)}
\bigg(\int_{\mathbb{D}}\frac{\omega(\zeta)}{|1-\zeta\bar{z}|^{pt+3p}}dA(\zeta)\bigg)^{\frac{1}{p}}
\asymp\frac{(1-|z|^{2})^{s+1}}{\sigma(z)}\frac{\widehat{\omega}(z)^{\frac{1}{p}}}{(1-|z|)^{t+3-\frac{1}{p}}}\asymp1.
\end{equation}
Together with Lemma \ref{L:bloch} and the boundedness of $T_{\mu,\omega}$, if $g\in\mathcal{B}$, we obtain
\begin{equation}\label{PR04}
\frac{(1-|z|^{2})^{s+1}}{\sigma(z)}\bigg|\int_{\mathbb{D}}\frac{\overline{g(u)}}{(1-u\bar{z})^{t+3}}d\mu(u)\bigg|
=|\langle T_{\mu,\omega}(h_{z}),g\rangle_{L^{2}(\omega)}|
\lesssim\|g\|_{\mathcal{B}}.
\end{equation}

Fix $r\in(0,1)$, without loss of generality, we consider an $r$-lattice $\{\xi_{j}\}_{j=1}^{\infty}$ such that $\{|\xi_j|\}_{j=1}^\infty$ is an increasing sequence and
$
\lim\limits_{j\to \infty}|\xi_j|=1.
$
Let $N$ be the positive integer such that
\[
D(\xi_{j},5r)\subseteq \bigg\{u:~\frac{3}{4}\leq |u|<1\bigg\},
\]
for any $j>N$.
Then there exists $r_{0}\in(0,1)$ such that
\begin{equation}\label{E:unr}
\bigcup\limits_{j=1}^{N}D(\xi_{j},5r)\subseteq D(0,r_{0}).
\end{equation}
Note that for any fixed $z\in\mathbb{D}$, the function
$
|1-\bar{z}u|^{-(2c+t+1)}
$
is subharmonic on $\mathbb{D}$. From Lemma \ref{prop-reverse}, we know $\kappa$ is regular. Recall that $\nu(z)=\sigma(z)(1-|z|)^{c+\frac{1}{p}-3}$. Together with Proposition \ref{la45}, for each $j>N$, we deduce
\begin{align}
\sup_{u\in D(\xi_{j},5r)}\frac{(1-|u|^{2})^{c-1}}{|1-\bar{z}u|^{2c+t+1}}\leq&
\sup_{u\in D(\xi_{j},5r)}\frac{(1-|u|^{2})^{c-1}}{\big(\frac{r(1-|u|^{2})}{2(1+r)}\big)^{2}}
\int_{B\big(u,\frac{r(1-|u|^{2})}{2(1+r)}\big)}\frac{1}{|1-\bar{z}\zeta|^{2c+t+1}}dA(\zeta)\nonumber\\
\lesssim&\sup_{u\in D(\xi_{j},5r)}\frac{1}{\kappa(u)(1-|u|)^{2}}
\int_{D(u,r)}\frac{\kappa(\zeta)(1-|\zeta|^{2})^{c-1}}{|1-\bar{z}\zeta|^{2c+t+1}}dA(\zeta)\nonumber\\
\lesssim&\frac{1}{\kappa(\xi_{j})(1-|\xi_{j}|)^{2}}
\int_{D(\xi_{j},6r)}\frac{\kappa(\zeta)(1-|\zeta|^{2})^{c-1}}{|1-\bar{z}\zeta|^{2c+t+1}}dA(\zeta)\nonumber\\
\asymp&\frac{1}{\widehat{\omega}(\xi_{j})^{\frac{1}{p}}(1-|\xi_{j}|)^{\frac{1}{p}}}
\int_{D(\xi_{j},6r)}\frac{\nu(\zeta)}{|1-\bar{z}\zeta|^{2c+t+1}}dA(\zeta).
\label{E:max}
\end{align}
For any $g\in\mathcal{B}$, it follows from Theorem 5.8 in \cite{zhu2007} or Proposition 2.4 in \cite{WL-2010} that for any distinct points $z$ and $u$ on $\mathbb{D}$,
\begin{equation}\label{E:bloch}
\frac{(1-|z|^{2})^{1-c}(1-|u|^{2})^{1-c}|g(z)-g(u)|}{|z-u||1-\bar{u}z|^{1-2c}}\lesssim\|g\|_{\mathcal{B}}.
\end{equation}
Let
\begin{equation}\label{E:ihat}
I(z)=\frac{(1-|z|^{2})^{s+1}}{\sigma(z)}J(z),
\end{equation}
where
\begin{equation}\label{QE:ihat}
J(z)=\int_{\mathbb{D}}\frac{|g(z)-g(u)|}{|1-\bar{u}z|^{t+3}}d\mu(u).
\end{equation}
From \eqref{E:bloch}, we obtain
\begin{align*}
J(z)
=&\frac{1}{(1-|z|^{2})^{1-c}}\int_{\mathbb{D}}\frac{(1-|z|^{2})^{1-c}
(1-|u|^{2})^{1-c}|g(z)-g(u)|}{|z-u||1-\bar{u}z|^{1-2c}}\frac{|z-u|(1-|u|^{2})^{c-1}}{|1-\bar{u}z|^{t+2c+2}}d\mu(u)\nonumber\\
\lesssim&\frac{\|g\|_{\mathcal{B}}}{(1-|z|^{2})^{1-c}}
\int_{\mathbb{D}}\frac{|z-u|}{|1-\bar{z}u|}\frac{(1-|u|^{2})^{c-1}}{|1-\bar{z}u|^{2c+t+1}}d\mu(u)\nonumber\\
<&\frac{\|g\|_{\mathcal{B}}}{(1-|z|^{2})^{1-c}}
\bigg(\sum_{j=1}^{N}\int_{D(\xi_{j},5r)}\frac{(1-|u|^{2})^{c-1}}{|1-\bar{z}u|^{2c+t+1}}d\mu(u)+
\sum_{j=N+1}^{\infty}\int_{D(\xi_{j},5r)}\frac{(1-|u|^{2})^{c-1}}{|1-\bar{z}u|^{2c+t+1}}d\mu(u)\bigg).
\end{align*}
Combining \eqref{E:unr} with \eqref{E:max}, we know
\begin{equation*}\label{QQE:ihat}
\sum_{j=1}^{N}\int_{D(\xi_{j},5r)}\frac{(1-|u|^{2})^{c-1}}{|1-\bar{z}u|^{2c+t+1}}d\mu(u)\leq
\int_{r_{0}\mathbb{D}}\frac{d\mu(u)}{(1-|u|)^{c+t+2}}
\end{equation*}
and
\begin{align*}\label{QQQE:ihat}
\sum_{j=N+1}^{\infty}\int_{D(\xi_{j},5r)}\frac{(1-|u|^{2})^{c-1}}{|1-\bar{z}u|^{2c+t+1}}d\mu(u)\leq&
\sum_{j=N+1}^{\infty}\mu(D(\xi_{j},5r))\sup_{u\in D(\xi_{j},5r)}\frac{(1-|u|^{2})^{c-1}}{|1-\bar{z}u|^{2c+t+1}}\nonumber\\
\lesssim&\sum_{j=N+1}^{\infty}\frac{\mu(D(\xi_{j},5r))}{\widehat{\omega}(\xi_{j})^{\frac{1}{p}}(1-|\xi_{j}|)^{\frac{1}{p}}}
\int_{D(\xi_{j},6r)}\frac{\nu(\zeta)dA(\zeta)}{|1-\bar{z}\zeta|^{2c+t+1}}.
\end{align*}
Together with \eqref{kappa}, \eqref{E:ihat} and \eqref{QE:ihat}, we have
\begin{align*}
I(z)
\lesssim&\|g\|_{\mathcal{B}}\bigg[\frac{\mu(r_{0}\mathbb{D})(1-|z|^{2})^{t-t_{0}+c}}{(1-r_{0})^{c+t+2}}+
\frac{(1-|z|^{2})^{s+c}}{\sigma(z)}
\sum_{j=N+1}^{\infty}\frac{\mu(D(\xi_{j},5r))}{\widehat{\omega}(\xi_{j})^{\frac{1}{p}}(1-|\xi_{j}|)^{\frac{1}{p}}}
\int_{D(\xi_{j},6r)}\frac{\nu(\zeta)dA(\zeta)}{|1-\bar{z}\zeta|^{2c+t+1}}\bigg].
\end{align*}
From Proposition \ref{P:regular}, we obtain that $\mu$ is a $\frac{1}{p}$-Carleson measure for $L_{a}^{1}(\omega)$.
By Lemma \ref{P:D} and \eqref{PR03}, we get
\begin{align*}
I(z)
\lesssim&\|g\|_{\mathcal{B}}\frac{\mu(r_{0}\mathbb{D})(1-|z|^{2})^{t-t_{0}+c}}{(1-r_{0})^{c+t+2}}+
\|g\|_{\mathcal{B}}\frac{(1-|z|^{2})^{s+c}}{\sigma(z)}
\int_{\mathbb{D}}\frac{\nu(\zeta)}{|1-\bar{z}\zeta|^{2c+t+1}}dA(\zeta)\nonumber\\
\lesssim&\|g\|_{\mathcal{B}}\frac{\mu(r_{0}\mathbb{D})}{(1-r_{0})^{c+t+2}}+
\|g\|_{\mathcal{B}}\frac{(1-|z|^{2})^{s+c}}{\sigma(z)}
\frac{\widehat{\nu}(z)}{(1-|z|)^{2c+t}}.
\end{align*}
In the case $0<p<1$, for each $z\in\mathbb{D}$, we obtain
\[
\frac{(1-|z|^{2})^{s+c}}{\sigma(z)}\frac{\widehat{\nu}(z)}{(1-|z|)^{2c+t}}\leq
\frac{(1-|z|^{2})^{s+c}}{\sigma(z)}\frac{\widehat{\omega}(z)^{\frac{1}{p}}\int_{|z|}^{1}(1-\rho)^{c+\frac{1}{p}-3}d\rho}{(1-|z|)^{2c+t}}
\asymp1.
\]
In the case $p=1$, from Corollary \ref{choose of c1}, we know that $\nu$ is regular. Then for each $z\in\mathbb{D}$, we get
\[
\frac{(1-|z|^{2})^{s+c}}{\sigma(z)}\frac{\widehat{\nu}(z)}{(1-|z|)^{2c+t}}
\asymp\frac{(1-|z|^{2})^{s+c}}{\sigma(z)}\frac{\nu(z)(1-|z|)}{(1-|z|)^{2c+t}}\asymp1.
\]
Therefore, for $0<p\leq1$, we conclude that
\begin{equation}\label{PR05}
I(z)\lesssim \|g\|_{\mathcal{B}},\quad \textmd{for~any}~z\in\mathbb{D}.
\end{equation}
Note that
\[\frac{(1-|z|^{2})^{s+1}}{\sigma(z)}\overline{g(z)T_{\mu,t+1}(1)(z)}=\langle T_{\mu,\omega}(h_{z}),g\rangle_{L^{2}(\omega)}
+\frac{(1-|z|^{2})^{s+1}}{\sigma(z)}\int_{\mathbb{D}}\frac{\overline{g(z)}-\overline{g(u)}}{(1-u\bar{z})^{t+3}}d\mu(u).
\]
Together with (\ref{PR04}), \eqref{E:ihat}, \eqref{QE:ihat} and (\ref{PR05}), for any $g\in\mathcal{B}$ and $z\in\mathbb{D}$, we obtain
\begin{equation}\label{PR06}
\frac{(1-|z|^{2})^{s+1}}{\sigma(z)}|g(z)T_{\mu,t+1}(1)(z)|\lesssim \|g\|_{\mathcal{B}}.
\end{equation}
Without loss of generality, assume that $g(0)=0$. Take the supremum over all $g\in\mathcal{B}$ with $\|g\|_\mathcal{B}\leq1$, by Theorem 5.7 in \cite{zhu2007} and (\ref{PR06}), we have
\begin{equation}\label{PR07}
\frac{(1-|z|^{2})^{s+1}}{\sigma(z)}\log\bigg(\frac{2}{1-|z|^{2}}\bigg)|T_{\mu,t+1}(1)(z)|\lesssim 1.
\end{equation}
Since $\mu$ is a $\frac{1}{p}$-Carleson measure for $L_{a}^{1}(\omega)$, for any $z\in\mathbb{D}$, we deduce from (\ref{PR02}) that
\begin{align}\label{PR08}
\frac{(1-|z|^{2})^{s+1}}{\sigma(z)}\log\bigg(\frac{2}{1-|z|^{2}}\bigg)|T_{\mu,t}(1)(z)|\leq&
\frac{(1-|z|^{2})^{s+1}}{\sigma(z)}\log\bigg(\frac{2}{1-|z|^{2}}\bigg)\int_{\mathbb{D}}\frac{1}{|1-\bar{z}u|^{t+2}}d\mu(u)
\nonumber\\\lesssim&\frac{(1-|z|^{2})^{s+1}}{\sigma(z)}\log\bigg(\frac{2}{1-|z|^{2}}\bigg)
\bigg(\int_{\mathbb{D}}\frac{\omega(u)}{|1-\bar{z}u|^{pt+2p}}dA(u)\bigg)^{\frac{1}{p}}
\nonumber\\\asymp&\frac{(1-|z|^{2})^{s+1}}{\sigma(z)}\log\bigg(\frac{2}{1-|z|^{2}}\bigg)
\frac{\widehat{\omega}(z)^{\frac{1}{p}}}{(1-|z|)^{t+2-\frac{1}{p}}}\nonumber\\
\lesssim&1.
\end{align}
Notice that
\begin{equation}\label{PR09}
(t+2)T_{\mu,t+1}(1)(z)=(t+2)T_{\mu,t}(1)(z)+z(T_{\mu,t}1)'(z).
\end{equation}
Combining (\ref{PR07}), (\ref{PR08}) with (\ref{PR09}), we conclude that $\|T_{\mu,t}(1)\|_{\mathcal{LB}^{s}_{\sigma}}<\infty$, which completes the proof.
\end{proof}

\vskip 0.1in

Next, we are in a position to complete the proof of Theorem \ref{T:regular}.
\begin{proof}[Proof of Theorem \ref{T:regular}]
By Proposition \ref{P:regular} and Proposition \ref{P1:regular}, it remains to show the sufficiency part.
From Corollary \ref{coro-Lemma}, there exist constants $C_{1}\in(0,1)$ and $C_{2}>1$ such that for any $C_{3}\in(0,C_{1}]$,
\begin{equation}\label{kappa1}
\kappa(r)\geq C_{2}^{-1}\widehat{\kappa}(0)(1-r)^{\frac{1}{C_{3}}-1}, \quad r\in[0,1),
\end{equation}
where
$
\kappa(z)=(1-|z|)^{-1}\widehat{\omega}(z).
$
Moreover, there exists a constant $C_4>1$ such that
\[
\int_{r}^{1}\omega(\rho)d\rho \leq C_{4}\int_{\frac{1+r}{2}}^{1}\omega(\rho)d\rho,\quad  r\in[0,1).
\]
By \eqref{E:nu}, Corollary \ref{choose of cp} and Corollary \ref{choose of c1}, there exists a constant $C_{5}>1$ such that the weight
\[
\nu(z)=\widehat{\omega}(z)^{\frac{1}{p}}(1-|z|)^{c+\frac{1}{p}-3}
\]
on $\mathbb{D}$ satisfies
\[
\int_{r}^{1}\nu(\rho)d\rho \leq C_{5}\int_{\frac{1+r}{2}}^{1}\nu(\rho)d\rho,
\quad r\in[0,1).
\]
By the hypothesis, there exists $t_{0}>0$ such that for each $t\geq t_{0}$,
\[
\|T_{\mu,t}(1)\|_{\mathcal{LB}^{s}_{\sigma}}<\infty,
\]
where $s=t-\frac{1}{p}+2$ and $\sigma^p(z)=\widehat{\omega}(z)$. Moreover, take a constant $C_{6}\in (0,C_1]$ such that
\[
t_1:=\frac{C_{6}+1}{pC_{6}}-2>\textmd{max}\bigg\{p^{-1}\log_{2}C_{4}+p^{-1}-2,~\log_{2}C_{5}-2c\bigg\}.
\]
Let
$$
t_{2}=\max\{t_{0},t_{1}\}.
$$
Next let $t\geq t_{2}$. According to Remark \ref{rema-doubling}, we have
\begin{equation}\label{R02}
\int_{\mathbb{D}}\frac{\omega(z)}{|1-\bar{z}\zeta|^{pt+2p}}dA(z) \asymp\frac{\widehat{\omega}(\zeta)}{(1-|\zeta|)^{pt+2p-1}}, \quad \zeta\in\mathbb{D}
\end{equation}
and
\begin{equation}\label{R03}
\int_{\mathbb{D}}\frac{\nu(z)}{|1-\bar{z}\zeta|^{2c+t+1}}dA(z) \asymp\frac{\widehat{\nu}(\zeta)}{(1-|\zeta|)^{2c+t}}, \quad \zeta\in\mathbb{D}.
\end{equation}
From Lemma \ref{L:bloch},
we know
\[
\|T_{\mu,\omega}(f)\|_{L_{a}^{1}(\omega)}=\sup_{\{g\in\mathcal{B}:\|g\|_{\mathcal{B}}=1\}}|\langle T_{\mu,\omega}(f),g \rangle_{L^{2}(\omega)}|.
\]
It remains to show that for any $f\in L^p_a(\omega)$ and $g\in\mathcal{B}$,
\begin{equation*}\label{newnew3}
|\langle T_{\mu,\omega}(f), g\rangle_{L^{2}(\omega)}|\lesssim\|f\|_{L_{a}^{p}(\omega)}\|g\|_{\mathcal{B}}.
\end{equation*}

It follows from Lemma \ref{prop-reverse} and \eqref{kappa1} that
\begin{equation}\label{E:NEW}
L_{a}^{p}(\omega)\subseteq L_{a}^{1}(dA_{t+1}).
\end{equation}
Let $P_{t+1}$ be the orthogonal projection from $L^2(dA_{t+1})$ to the weighted Bergman space $L^2_a(dA_{t+1})$. For any $f\in L^p_a(\omega)$ and $g\in\mathcal{B}$, we have
\begin{align}\label{E:NEW1}
\langle T_{\mu,\omega}(f), g\rangle_{L^{2}(\omega)}=&\int_{\mathbb{D}}f(z)\overline{g(z)}d\mu(z)\nonumber\\
=&\underbrace{\int_{\mathbb{D}}P_{t+1}(f\bar{g})(z)d\mu(z)}_{I_1}+
\underbrace{\bigg(\int_{\mathbb{D}}(f\bar{g})(z)d\mu(z)-\int_{\mathbb{D}}P_{t+1}(f\bar{g})(z)d\mu(z)\bigg)}_{I_2}.
\end{align}
From Lemma \ref{prop-reverse} and Proposition \ref{la45}, for any $f\in L_{a}^{p}(\omega)$ and $r\in(0,1)$, we obtain
\begin{align}\label{pro-sub}
|f(\zeta)|^{p}\leq&\frac{1}{\big(\frac{r(1-|\zeta|^{2})}{2(1+r)}\big)^{2}}
\int_{B\big(\zeta,\frac{r(1-|\zeta|^{2})}{2(1+r)}\big)}|f(u)|^{p}dA(u)\nonumber\\
\lesssim&\frac{1}{\kappa(\zeta)(1-|\zeta|^{2})^{2}}\int_{D(\zeta,r)}|f(u)|^{p}\kappa(u)dA(u)\nonumber\\
\lesssim&\frac{\|f\|_{L_{a}^{p}(\omega)}^{p}}{\widehat{\omega}(\zeta)(1-|\zeta|)},\quad \zeta\in\mathbb{D}.
\end{align}
Recall that $s=t-\frac{1}{p}+2$ and $\sigma^p(z)=\widehat{\omega}(z)$. According to \eqref{pro-sub}, we deduce
\begin{align}\label{new1}
|I_{1}|
\leq&\int_{\mathbb{D}}|f(\zeta)||g(\zeta)|
\cdot\bigg|\int_{\mathbb{D}}K_{z}^{t+1}(\zeta)d\mu(z)\bigg|dA_{t+1}(\zeta)\nonumber\\
=&\int_{\mathbb{D}}|f(\zeta)|^{p}|f(\zeta)|^{1-p}|g(\zeta)|
\cdot\bigg|\int_{\mathbb{D}}K_{z}^{t+1}(\zeta)d\mu(z)\bigg|dA_{t+1}(\zeta)\nonumber\\
\lesssim&\|f\|_{L_{a}^{p}(\omega)}^{1-p}\int_{\mathbb{D}}|f(\zeta)|^{p}\kappa(\zeta)
\underbrace
{\frac{(1-|\zeta|^{2})^{s+1}}{\sigma(\zeta)}|g(\zeta)|\bigg|
\int_{\mathbb{D}}K_{z}^{t+1}(\zeta)d\mu(z)\bigg|}_{M(\zeta)}dA(\zeta),
\end{align}
where
\[
K_{z}^{t+1}(\zeta)=\frac{1}{(1-\bar{z}\zeta)^{t+3}}
\]
is the reproducing kernel of $L_{a}^2(dA_{t+1})$. Rewrite
\begin{equation}\label{CRR2}
M(\zeta)
=\frac{(1-|\zeta|^{2})^{s+1}}{\sigma(\zeta)}|g(\zeta)|\bigg|
\int_{\mathbb{D}}\frac{1-\bar{z}\zeta}{(1-\bar{z}\zeta)^{t+3}}d\mu(z)+
\int_{\mathbb{D}}\frac{\bar{z}\zeta}{(1-\bar{z}\zeta)^{t+3}}d\mu(z)\bigg|.
\end{equation}
Since $\|T_{\mu,t}(1)\|_{\mathcal{LB}^{s}_{\sigma}}<\infty$ and $\mu$ is a $\frac{1}{p}$-Carleson measure for $L_{a}^{1}(\omega)$, by Lemma 2.1 in \cite{WL-2010} and (\ref{R02}), for any $\zeta\in\mathbb{D}$, we have
\begin{align*}
M(\zeta)
\lesssim&|g(\zeta)|\frac{(1-|\zeta|^{2})^{s+1}}{\sigma(\zeta)}
\bigg(\int_{\mathbb{D}}\frac{\omega(z)}{|1-z\bar{\zeta}|^{pt+2p}}dA(z)\bigg)^{\frac{1}{p}}
+|g(\zeta)|\frac{(1-|\zeta|^{2})^{s+1}}{\sigma(\zeta)}|(T_{\mu,t}1)'(\zeta)|\nonumber\\
\lesssim&\log\bigg(\frac{2}{1-|\zeta|^{2}}\bigg)\|g\|_{\mathcal{B}}\frac{(1-|\zeta|^{2})^{s+1}}{\sigma(\zeta)}
\frac{\widehat{\omega}(\zeta)^{\frac{1}{p}}}{(1-|\zeta|)^{t+2-\frac{1}{p}}}\nonumber\\
+&\log\bigg(\frac{2}{1-|\zeta|^{2}}\bigg)\|g\|_{\mathcal{B}}\frac{(1-|\zeta|^{2})^{s+1}}{\sigma(\zeta)}
|(T_{\mu,t}1)'(\zeta)|\nonumber\\
\lesssim&\|g\|_{\mathcal{B}}+\|g\|_{\mathcal{B}}\|T_{\mu,t}(1)\|_{\mathcal{LB}^{s}_{\sigma}}
\lesssim \|g\|_{\mathcal{B}}.
\end{align*}
Together with Lemma \ref{prop-reverse} and \eqref{new1}, we conclude that
\begin{equation}\label{E:I1}
|I_{1}|\lesssim \|f\|_{L_{a}^{p}(\omega)}\|g\|_{\mathcal{B}}.
\end{equation}

By \eqref{E:NEW} and \eqref{pro-sub}, we deduce
\begin{align}\label{new3}
|I_2|\leq&\int_{\mathbb{D}}|f(\zeta)|\int_{\mathbb{D}}|g(z)-g(\zeta)||K_{z}^{t+1}(\zeta)|d\mu(z)dA_{t+1}(\zeta)
\nonumber\\
=&\int_{\mathbb{D}}|f(\zeta)|^{p}|f(\zeta)|^{1-p}
\int_{\mathbb{D}}|g(z)-g(\zeta)||K_{z}^{t+1}(\zeta)|d\mu(z)dA_{t+1}(\zeta)\nonumber\\
\lesssim&\|f\|_{L_{a}^{p}(\omega)}^{1-p}\int_{\mathbb{D}}|f(\zeta)|^{p}\kappa(\zeta)\underbrace{\frac{(1-|\zeta|^{2})^{s+1}}
{\sigma(\zeta)}\int_{\mathbb{D}}\frac{|g(z)-g(\zeta)|}{|1-\bar{z}\zeta|^{t+3}}d\mu(z)}_{N(\zeta)}dA(\zeta).
\end{align}
Using a similar argument as in \eqref{PR05}, we obtain
\[
N(\zeta)\lesssim \|g\|_{\mathcal{B}}, \quad \textmd{for~any}~\zeta\in\mathbb{D}.
\]
It follows from Lemma \ref{prop-reverse} that
\begin{equation}\label{E:I2}
|I_{2}|\lesssim \|f\|_{L_{a}^{p}(\omega)}\|g\|_{\mathcal{B}}.
\end{equation}
Combining \eqref{E:NEW1}, \eqref{E:I1} with \eqref{E:I2}, we conclude that $T_{\mu,\omega}:L_{a}^{p}(\omega)\to L_{a}^{1}(\omega)$ is bounded,
which completes the whole proof.
\end{proof}

\section{Compactness of Toeplitz operators $T_{\mu,\omega}$ on $L_{a}^{p}(\omega)$}
In this section, we study the compact Toeplitz operators between $L^p_a(\omega)$ and $L^1_a(\omega)$ for $0<p\leq1$ and a $\mathcal{D}$ weight $\omega$. For $0<p,q<\infty$, the Toeplitz operator $T_{\mu,\omega}:L_{a}^{p}(\omega)\rightarrow L_{a}^{q}(\omega)$ is compact if and only if
\[
\lim\limits_{j\rightarrow\infty}\|T_{\mu,\omega}(f_{j})\|_{L_{a}^{q}(\omega)}=0,
\]
whenever $\{f_{j}\}_{j=1}^{\infty}$ is a bounded sequence in $L_{a}^{p}(\omega)$ such that $f_{j}$ uniformly converges to 0 on each compact subset of $\mathbb{D}$.

We also need the following geometrical characterization of vanishing $q$-Carleson measure for $L_a^p(\omega)$.

\begin{lemma}\label{P:D1} Let $\mu$ be a positive Borel measure, $\omega$ be a $\mathcal{D}$ weight and $0<p\leq q<\infty$. Then $\mu$ is a vanishing $q$-Carleson measure for $L_{a}^{p}(\omega)$ if and only if
\[
\lim_{|z|\rightarrow1^{-}}\frac{\mu(D(z,r))}{\widehat{\omega}(z)^{\frac{q}{p}}(1-|z|)^{\frac{q}{p}}}=0,\quad\textmd{for~some}~r\in(0,1).
\]
\end{lemma}


In the following, we proceed to establish the corresponding result for the compactness of $T_{\mu,\omega}$
by modifying the proofs given in Theorem \ref{T:regular}.

\begin{proof}[Proof of Theorem \ref{RC}] Necessity.
Let $0<p\leq 1$. Consider a real number $\alpha$ with $\alpha p>1$ and let $\beta=\alpha+1-\frac{1}{p}$. For $z\in\mathbb{D}$, let
\[
f_{z}(\zeta)=\frac{(K_{z}^{\omega}(\zeta))^{\alpha}}{\|K_{z}^{\omega}\|^{\beta}_{L^{\beta}(\omega)}} ~\text{ and }~ g_{z}(\zeta)=\frac{(K_{z}^{\omega}(\zeta))^{\alpha}}{\|K_{z}^{\omega}\|^{\alpha+1}_{L^{\alpha+1}(\omega)}}, \quad \zeta\in\mathbb{D}.
\]
From \cite[Formula (20)]{PR1}, we have
\[
|K_{z}^{\omega}(\zeta)|\leq\sum_{n=0}^{\infty}\frac{|\zeta|^{n}}{2\omega_{2n+1}}
\asymp\frac{1}{\omega_{1}}+\int_{0}^{|\zeta|}\frac{1}{\widehat{\omega}(t)(1-t)^{2}}dt\lesssim\frac{1}{\widehat{\omega}(\zeta)(1-|\zeta|)},\quad |\zeta|\to 1^{-},
\]
where $\omega_{2n+1}=\int_{0}^{1}\omega(r)r^{2n+1}dr$.
Then we obtain
\[
|f_{z}(\zeta)|\asymp \widehat{\omega}(z)^{\beta-1}(1-|z|)^{\beta-1}|K_{z}^{\omega}(\zeta)|^{\alpha}
\lesssim\frac{\widehat{\omega}(z)^{\beta-1}(1-|z|)^{\beta-1}}{\widehat{\omega}(\zeta)^{\alpha}(1-|\zeta|)^{\alpha}},\quad \zeta\in\mathbb{D}.
\]
Hence, for any compact subset $E\subset\mathbb{D}$, it holds that
\[
\lim_{|z|\to 1^-}\sup_{\zeta\in E}|f_{z}(\zeta)|=0.
\]
Moreover, by \eqref{rp1}, we get
\[
\sup_{z\in\mathbb{D}}\|f_z\|_{L^p(\omega)}<\infty.
\]
Then
\begin{equation}\label{E:vanish}
\lim\limits_{|z|\to 1^-}\|T_{\mu,\omega}(f_{z})\|_{L_{a}^{1}(\omega)}=0.
\end{equation}
According to \eqref{rp3}, there exist constants
$C_{1}>0$ and $r\in (0,1)$ such that for any $z\in\mathbb{D}$,
\[
\frac{\mu(D(z,r))}{\widehat{\omega}(z)^{\frac{1}{p}}(1-|z|)^{\frac{1}{p}}}\leq C_{1}
|\langle T_{\mu,\omega}(f_z),g_z\rangle_{L^{2}(\omega)}|.
\]
It follows from \eqref{rpnn}, \eqref{E:vanish} and Lemma \ref{P:D1} that $\mu$ is a vanishing $\frac{1}{p}$-Carleson measure for $L_a^1(\omega)$.

Let $\nu$ be the weight given by \eqref{E:nu}.
According to Corollary \ref{choose of cp}, Remark \ref{rema-doubling} and Corollary \ref{choose of c1}, we can choose a positive constant $C_{2}$ such that for each
\[
\gamma\geq\frac{C_{2}+1}{pC_{2}}-2>0,
\]
we have
\begin{equation}\label{PR02-1}
\int_{\mathbb{D}}\frac{\omega(z)}{|1-\bar{z}\zeta|^{p\gamma+2p}}dA(z) \asymp\frac{\widehat{\omega}(\zeta)}{(1-|\zeta|)^{p\gamma+2p-1}}, \quad \zeta\in\mathbb{D}
\end{equation}
and
\begin{equation}\label{PR03-1}
\int_{\mathbb{D}}\frac{\nu(z)}{|1-\bar{z}\zeta|^{2c+\gamma+1}}dA(z) \asymp\frac{\widehat{\nu}(\zeta)}{(1-|\zeta|)^{2c+\gamma}}, \quad \zeta\in\mathbb{D}.
\end{equation}
Let
$
\kappa(z)=(1-|z|)^{-1}\widehat{\omega}(z).
$
It follows from Corollary \ref{coro-Lemma} that there exist constants $C_{3}\in(0,1)$ and $C_{4}>1$ such that for any constant $C_{5}\in(0,C_{3}]$,
\begin{equation}\label{kappa2}
\kappa(r)\geq C_{4}^{-1}\widehat{\kappa}(0)(1-r)^{\frac{1}{C_{5}}-1}, \quad r\in[0,1).
\end{equation}
Choose a positive constant $C_{6}$ with $C_{6}<\min\{C_{2},C_{3}\}$ and let
\begin{equation}\label{C6}
t_{0}=\frac{C_{6}+1}{pC_{6}}-2.
\end{equation}
For each $t\geq t_{0}$ and $z\in\mathbb{D}$, let
\[
h_{z}(\zeta)=\frac{(1-|z|^{2})^{s+1}}{\sigma(z)}
\frac{1}{(1-\zeta\bar{z})^{t+3}}, \quad \zeta\in\mathbb{D},
\]
where $s=t-\frac{1}{p}+2$ and $\sigma^p(z)=\widehat{\omega}(z)$.
With similar arguments as in \eqref{E:testh}, for $0<p\leq 1$, we have
\[
\|h_{z}\|_{L_{a}^{p}(\omega)}\asymp 1.
\]
Moreover, by \eqref{kappa2}, we obtain
\[
|h_{z}(\zeta)|\leq\frac{(1-|z|^{2})^{s+1}}{\sigma(z)}\frac{1}{(1-|\zeta|)^{t+3}}
\lesssim \frac{(1-|z|^{2})^{t-\frac{1}{p}+3}}{(1-|z|)^{\frac{1}{pC_{6}}}}\frac{1}{(1-|\zeta|)^{t+3}}\asymp\frac{(1-|z|^{2})^{t-t_{0}+1}}{(1-|\zeta|)^{t+3}}.
\]
Thus, for any compact subset $E\subset\mathbb{D}$, it holds that
\[
\lim_{|z|\to 1^-}\sup_{\zeta\in E}|h_{z}(\zeta)|=0.
\]
By the compactness of $T_{\mu,\omega}$, we have
\begin{equation}\label{RC01}
\lim\limits_{|z|\rightarrow1^{-}}\|T_{\mu,\omega}(h_{z})\|_{L_{a}^{1}(\omega)}=0.
\end{equation}

For a fixed $\delta\in(0,1)$, without loss of generality, we consider a $\delta$-lattice $\{\xi_{i}\}_{i=1}^{\infty}$ such that $\{|\xi_i|\}_{i=1}^\infty$ is an increasing sequence and
$
\lim\limits_{i\to \infty}|\xi_i|=1.
$
Note that $\mu$ is a vanishing $\frac{1}{p}$-Carleson measure for $L_{a}^{1}(\omega)$. From Lemma \ref{P:D1}, for any given $\varepsilon>0$, there exists $k_{1}\in\mathbb{N}$ such that for any $i>k_{1}$,
\begin{equation}\label{RC02}
\frac{\mu(D(\xi_{i},5\delta))}{\widehat{\omega}(\xi_{i})^{\frac{1}{p}}(1-|\xi_{i}|)^{\frac{1}{p}}}<\varepsilon.
\end{equation}
Notice that there exists $r_1\in(0,1)$ such that
\begin{equation}\label{RC02v}
\mathbb{D}\backslash r_1\mathbb{D}\subseteq\bigcup_{i=k_{1}+1}^{\infty}D(\xi_{i},5\delta).
\end{equation}
Let
\begin{equation*}\label{CE:ihat}
I(z)=\frac{(1-|z|^{2})^{s+1}}{\sigma(z)}\int_{\mathbb{D}}\frac{|g(z)-g(u)|}
{|1-u\bar{z}|^{t+3}}d\mu(u).
\end{equation*}
By \eqref{E:bloch}, we have
\begin{align*}\label{CE:ihat1}
I(z)
\lesssim&\|g\|_{\mathcal{B}}\frac{(1-|z|^{2})^{s+c}}{\sigma(z)}
\bigg(\int_{r_1\mathbb{D}}\frac{(1-|u|^{2})^{c-1}}{|1-\bar{z}u|^{2c+t+1}}d\mu(u)+
\int_{\mathbb{D}\backslash r_1\mathbb{D}}\frac{(1-|u|^{2})^{c-1}}{|1-\bar{z}u|^{2c+t+1}}d\mu(u)\bigg).
\end{align*}
From \eqref{kappa2}, we obtain
\begin{equation*}\label{CRD1}
\frac{(1-|z|^{2})^{s+c}}{\sigma(z)}\int_{r_1\mathbb{D}}\frac{(1-|u|^{2})^{c-1}}{|1-\bar{z}u|^{2c+t+1}}d\mu(u)
\lesssim\frac{(1-|z|^{2})^{s+c}}{(1-|z|^{2})^{\frac{1}{pC_{6}}}}\int_{r_1\mathbb{D}}\frac{d\mu(u)}{(1-|u|)^{c+t+2}}
\lesssim\frac{(1-|z|)^{t-t_{0}+c}\mu(r_1\mathbb{D})}{(1-r_1)^{c+t+2}}.
\end{equation*}
Using the same estimates as in \eqref{E:max}, for any $i>k_{1}$, we have
\begin{equation}\label{recall-1}
\sup_{u\in D(\xi_{i},5\delta)}\frac{(1-|u|^{2})^{c-1}}{|1-\bar{z}u|^{2c+t+1}}\lesssim
\frac{1}{\widehat{\omega}(\xi_{i})^{\frac{1}{p}}(1-|\xi_{i}|)^{\frac{1}{p}}}
\int_{D(\xi_{i},6\delta)}\frac{\nu(u)}{|1-\bar{z}u|^{2c+t+1}}dA(u).
\end{equation}
Combining \eqref{PR03-1}, \eqref{RC02}, \eqref{RC02v} with \eqref{recall-1}, for $z\in\mathbb{D}$, we get
\begin{align*}
\int_{\mathbb{D}\backslash r_1 \mathbb{D}}\frac{(1-|u|^{2})^{c-1}}{|1-\bar{z}u|^{2c+t+1}}d\mu(u)
\leq&\sum_{i=k_{1}+1}^{\infty}\mu(D(\xi_{i},5\delta))
\sup_{u\in D(\xi_{i},5\delta)}\frac{(1-|u|^{2})^{c-1}}{|1-\bar{z}u|^{2c+t+1}}\nonumber\\
\lesssim&\sum_{i=k_{1}+1}^{\infty}\frac{\mu(D(\xi_{i},5\delta))}{\widehat{\omega}(\xi_{i})^{\frac{1}{p}}(1-|\xi_{i}|)^{\frac{1}{p}}}
\int_{D(\xi_{i},6\delta)}\frac{\nu(u)}{|1-\bar{z}u|^{2c+t+1}}dA(u)\nonumber\\
\lesssim&\varepsilon\int_{\mathbb{D}}\frac{\nu(u)}{|1-\bar{z}u|^{2c+t+1}}dA(u)\nonumber\\
\asymp&\varepsilon\frac{\widehat{\nu}(z)}{(1-|z|)^{2c+t}}.
\end{align*}
Moreover, for $0<p\leq1$, we have
\begin{equation*}
\sup_{z\in\mathbb{D}}\frac{(1-|z|^{2})^{s+c}}{\sigma(z)}
\frac{\widehat{\nu}(z)}{(1-|z|)^{2c+t}}<+\infty.
\end{equation*}
Then we conclude that
\begin{equation}\label{RCCCC}
I(z)\lesssim\varepsilon\|g\|_{\mathcal{B}}, \quad \textmd{for}~|z|\geq 1-(1-r_{1})^{\frac{c+t+2}{t-t_{0}+c}}\mu(r_{1}\mathbb{D})^{-\frac{1}{t-t_{0}+c}}\varepsilon^{\frac{1}{t-t_{0}+c}}.
\end{equation}
Note that
\begin{equation*}\label{E:NEW3}
\langle T_{\mu,\omega}(h_{z}),g\rangle_{L^{2}(\omega)}=\frac{(1-|z|^{2})^{s+1}}{\sigma(z)}
\int_{\mathbb{D}}\frac{\overline{g(u)}}{(1-u\bar{z})^{t+3}}d\mu(u).
\end{equation*}
Then we get
\begin{equation*}
\frac{(1-|z|^{2})^{s+1}}{\sigma(z)}\overline{g(z)T_{\mu,t+1}(1)(z)}
=\langle T_{\mu,\omega}(h_{z}),g\rangle_{L^{2}(\omega)}
+\frac{(1-|z|^{2})^{s+1}}{\sigma(z)}\int_{\mathbb{D}}\frac{\overline{g(z)}-\overline{g(u)}}{(1-u\bar{z})^{t+3}}d\mu(u).
\end{equation*}
Together with \eqref{RC01}, \eqref{RCCCC} and Theorem 5.7 in \cite{zhu2007}, we obtain
\begin{equation}\label{RC05}
\lim_{|z|\rightarrow1^{-}}\frac{(1-|z|^{2})^{s+1}}{\sigma(z)}\log\bigg(\frac{2}{1-|z|^{2}}\bigg)|T_{\mu,t+1}(1)(z)|=0.
\end{equation}
Since $\mu$ is a vanishing $\frac{1}{p}$-Carleson measure for $L_{a}^{1}(\omega)$, by a similar calculation as in \eqref{PR08},
we obtain
\begin{equation}\label{RC06-}
\lim_{|z|\rightarrow1^{-}}\frac{(1-|z|^{2})^{s+1}}{\sigma(z)}\log\bigg(\frac{2}{1-|z|^{2}}\bigg)|T_{\mu,t}(1)(z)|=0.
\end{equation}
Combining \eqref{RC05}, \eqref{RC06-} with the following identity
\begin{equation*}
(t+2)T_{\mu,t+1}(1)(z)=(t+2)T_{\mu,t}(1)(z)+z(T_{\mu,t}1)'(z),
\end{equation*}
we conclude that
\[
\lim_{|z|\to 1^-}\frac{(1-|z|^{2})^{s+1}}{\sigma(z)}\log\bigg(\frac{2}{1-|z|^{2}}\bigg)|(T_{\mu,t}1)'(z)|=0.
\]

Sufficiency. From the hypothesis, we know there exists $t_{1}>0$ such that for each $t\geq t_{1}$,
\begin{equation}\label{C:condition}
\lim_{|z|\to 1^-}\frac{(1-|z|^{2})^{s+1}}{\sigma(z)}\log\bigg(\frac{2}{1-|z|^{2}}\bigg)|(T_{\mu,t}1)'(z)|=0,
\end{equation}
where $s=t-\frac{1}{p}+2$ and $\sigma^p(z)=\widehat{\omega}(z)$.
Let
\[
t_{2}=\max\bigg\{\frac{C_{6}+1}{pC_{6}}-2,t_{1}\bigg\},
\]
where $C_{6}$ is the constant in \eqref{C6}.
Next, let
\[
t\geq t_{2}.
\]
It follows from Lemma \ref{prop-reverse} and \eqref{kappa2} that
\begin{equation*}
L_{a}^{p}(\omega)\subseteq L_{a}^{1}(dA_{t+1}).
\end{equation*}
Let $\{f_{j}\}_{j=1}^{\infty}$ be a sequence in $L_{a}^{p}(\omega)$ such that $\|f_{j}\|_{L_{a}^{p}(\omega)}\leq1$ and $f_{j}$ uniformly converges to $0$ on any compact subset of $\mathbb{D}$.
From Lemma \ref{L:bloch}, for any $g\in\mathcal{B}$, we have
\begin{align}\label{decomposition}
\langle T_{\mu,\omega}(f_{j}), g\rangle_{L^{2}(\omega)}=&\int_{\mathbb{D}}f_{j}(z)\overline{g(z)}d\mu(z)\nonumber\\
=&\underbrace{\int_{\mathbb{D}}P_{t+1}(f_{j}\bar{g})(z)d\mu(z)}_{I_{1,j}}+
\underbrace{\bigg(\int_{\mathbb{D}}(f_{j}\bar{g})(z)d\mu(z)-\int_{\mathbb{D}}P_{t+1}(f_{j}\bar{g})(z)d\mu(z)\bigg)}_{I_{2,j}}.
\end{align}
By using arguments analogous to that used in \eqref{new1} and \eqref{CRR2}, together with \cite[Lemma 2.1]{WL-2010}, we get
\begin{align*}\label{new2}
|I_{1,j}|
\lesssim&\|g\|_{\mathcal{B}}\|f_{j}\|_{L_{a}^{p}(\omega)}^{1-p}
\underbrace{\int_{\mathbb{D}}|f_{j}(\zeta)|^{p}\kappa(\zeta)\frac{(1-|\zeta|^{2})^{s+1}}{\sigma(\zeta)}
\log\bigg(\frac{2}{1-|\zeta|^{2}}\bigg)\int_{\mathbb{D}}\frac{1}{|1-\bar{z}\zeta|^{t+2}}d\mu(z)dA(\zeta)}_{I_{1,j}^{1}}\nonumber\\
+&\|g\|_{\mathcal{B}}\|f_{j}\|_{L_{a}^{p}(\omega)}^{1-p}
\underbrace{\int_{\mathbb{D}}|f_{j}(\zeta)|^{p}\kappa(\zeta)\frac{(1-|\zeta|^{2})^{s+1}}{\sigma(\zeta)}
\log\bigg(\frac{2}{1-|\zeta|^{2}}\bigg)|(T_{\mu,t}1)'(\zeta)|dA(\zeta)}_{I_{1,j}^{2}}.
\end{align*}
Since $\mu$ is a vanishing $\frac{1}{p}$-Carleson measure for $L^{1}_{a}(\omega)$, combine \eqref{C:condition} with a similar argument as in \eqref{PR08}, for any given $\varepsilon>0$, there exists $r_{2}\in(0,1)$ such that for any $\zeta\in\mathbb{D}$ with $r_{2}\leq|\zeta|<1$,
\begin{equation}\label{C:NN3}
\frac{(1-|\zeta|^{2})^{s+1}}{\sigma(\zeta)}\log\bigg(\frac{2}{1-|\zeta|^{2}}\bigg)
\int_{\mathbb{D}}\frac{1}{|1-\bar{z}\zeta|^{t+2}}d\mu(z)<\varepsilon
\end{equation}
and
\begin{equation}\label{C:NN4}
\frac{(1-|\zeta|^{2})^{s+1}}{\sigma(\zeta)}\log\bigg(\frac{2}{1-|\zeta|^{2}}\bigg)|(T_{\mu,t}1)'(\zeta)|<\varepsilon.
\end{equation}
There also exists $k_{2}\in\mathbb{N}$ such that for any $j>k_{2}$, $|f_{j}|^{p}<\varepsilon$ on $r_{2}\mathbb{D}$.
Notice that $\mu$ is a $\frac{1}{p}$-Carleson measure for $L_{a}^{1}(\omega)$, together with \eqref{C:condition}, we get
\begin{equation}\label{C:NN1}
\sup_{\zeta\in \mathbb{D}}\bigg[\frac{(1-|\zeta|^{2})^{s+1}}
{\sigma(\zeta)}\log\bigg(\frac{2}{1-|\zeta|^{2}}\bigg)
\int_{\mathbb{D}}\frac{1}{|1-\bar{z}\zeta|^{t+2}}d\mu(z)\bigg]<\infty
\end{equation}
and
\begin{equation}\label{C:NN2}
\sup_{\zeta\in \mathbb{D}}
\bigg[\frac{(1-|\zeta|^{2})^{s+1}}{\sigma(\zeta)}\log\bigg(\frac{2}{1-|\zeta|^{2}}\bigg)
|(T_{\mu,t}1)'(\zeta)|\bigg]<\infty.
\end{equation}
Combining
Lemma \ref{prop-reverse}, \eqref{C:NN3}, \eqref{C:NN4}, \eqref{C:NN1} with \eqref{C:NN2},
for any $j>k_{2}$, we obtain
\begin{align*}
I_{1,j}^{1}=&
\int_{r_{2}\mathbb{D}}|f_{j}(\zeta)|^{p}\kappa(\zeta)\frac{(1-|\zeta|^{2})^{s+1}}{\sigma(\zeta)}
\log\bigg(\frac{2}{1-|\zeta|^{2}}\bigg)\int_{\mathbb{D}}\frac{1}{|1-\bar{z}\zeta|^{t+2}}d\mu(z)dA(\zeta)\nonumber\\
+&\int_{\mathbb{D}\backslash r_{2}\mathbb{D}}|f_{j}(\zeta)|^{p}\kappa(\zeta)
\frac{(1-|\zeta|^{2})^{s+1}}{\sigma(\zeta)}
\log\bigg(\frac{2}{1-|\zeta|^{2}}\bigg)\int_{\mathbb{D}}\frac{1}{|1-\bar{z}\zeta|^{t+2}}d\mu(z)dA(\zeta)\nonumber\\
\lesssim&\varepsilon \sup_{\zeta\in \mathbb{D}}\bigg[\frac{(1-|\zeta|^{2})^{s+1}}
{\sigma(\zeta)}\log\bigg(\frac{2}{1-|\zeta|^{2}}\bigg)
\int_{\mathbb{D}}\frac{1}{|1-\bar{z}\zeta|^{t+2}}d\mu(z)\bigg]+\varepsilon\|f_{j}\|^{p}_{L_{a}^{p}(\omega)}\lesssim\varepsilon
\end{align*}
and
\begin{align*}
I_{1,j}^{2}=&
\int_{r_{2}\mathbb{D}}|f_{j}(\zeta)|^{p}\kappa(\zeta)\frac{(1-|\zeta|^{2})^{s+1}}{\sigma(\zeta)}
\log\bigg(\frac{2}{1-|\zeta|^{2}}\bigg)|(T_{\mu,t}1)'(\zeta)|dA(\zeta)\nonumber\\
+&\int_{\mathbb{D}\backslash r_{2}\mathbb{D}}|f_{j}(\zeta)|^{p}\kappa(\zeta)\frac{(1-|\zeta|^{2})^{s+1}}{\sigma(\zeta)}
\log\bigg(\frac{2}{1-|\zeta|^{2}}\bigg)|(T_{\mu,t}1)'(\zeta)|dA(\zeta)\nonumber\\
\lesssim&\varepsilon \sup_{\zeta\in \mathbb{D}}
\bigg[\frac{(1-|\zeta|^{2})^{s+1}}{\sigma(\zeta)}\log\bigg(\frac{2}{1-|\zeta|^{2}}\bigg)
|(T_{\mu,t}1)'(\zeta)|\bigg]+\varepsilon\|f_{j}\|^{p}_{L_{a}^{p}(\omega)}\lesssim\varepsilon.
\end{align*}
Thus, we conclude that
\begin{equation}\label{RC10}
|I_{1,j}|\lesssim\varepsilon\|g\|_{\mathcal{B}}, \quad \textmd{for}~j>k_{2}.
\end{equation}
Using the similar argument as in \eqref{new3}, we have
\begin{equation}\label{new4}
|I_{2,j}|\lesssim
\|f_{j}\|^{1-p}_{L_{a}^{p}(\omega)}
\int_{\mathbb{D}}|f_{j}(\zeta)|^{p}\kappa(\zeta)\frac{(1-|\zeta|^{2})^{s+1}}{\sigma(\zeta)}
\int_{\mathbb{D}}\frac{|g(u)-g(\zeta)|}{|1-\bar{\zeta}u|^{t+3}}d\mu(u)dA(\zeta).
\end{equation}
By adopting the same procedure as in \eqref{PR05}, we can carry out
\begin{equation}\label{E:I2-C2}
\sup_{\zeta\in\mathbb{D}}\bigg(\frac{(1-|\zeta|^{2})^{s+1}}{\sigma(\zeta)}
\int_{\mathbb{D}}\frac{|g(u)-g(\zeta)|}{|1-\bar{\zeta}u|^{t+3}}d\mu(u)\bigg)
\lesssim \|g\|_{\mathcal{B}}.
\end{equation}
Note that there exists $r_{3}\in(0,1)$ such that for any $\zeta\in\mathbb{D}$ with $r_{3}\leq|\zeta|<1$,
\[
(1-|\zeta|)^{t-t_{2}+c}<\frac{(1-r_{1})^{c+t+2}}{\mu(r_{1}\mathbb{D})}\varepsilon,
\]
where $r_{1}$ is the one in \eqref{RC02v}. Then there also exists $k_{3}\in\mathbb{N}$ such that for any $j>k_{3}$, $|f_{j}|^{p}<\varepsilon$ on $r_{3}\mathbb{D}$.
The same reasoning as in \eqref{RCCCC} gives that for any $\zeta\in\mathbb{D}\backslash r_{3}\mathbb{D}$, it holds that
\begin{equation}\label{E:I2-C3}
\frac{(1-|\zeta|^{2})^{s+1}}{\sigma(\zeta)}
\int_{\mathbb{D}}\frac{|g(u)-g(\zeta)|}{|1-\bar{\zeta}u|^{t+3}}d\mu(u)\lesssim\varepsilon\|g\|_{\mathcal{B}}.
\end{equation}
According to Lemma \ref{prop-reverse}, \eqref{new4}, \eqref{E:I2-C2} and \eqref{E:I2-C3}, for $j>k_{3}$, we obtain
\begin{align*}
|I_{2,j}|\lesssim&
\|f_{j}\|^{1-p}_{L_{a}^{p}(\omega)}\int_{r_{3}\mathbb{D}}|f_{j}(\zeta)|^{p}\kappa(\zeta)\frac{(1-|\zeta|^{2})^{s+1}}{\sigma(\zeta)}
\int_{ \mathbb{D}}\frac{|g(u)-g(\zeta)|}{|1-\bar{\zeta}u|^{t+3}}d\mu(u)dA(\zeta)\nonumber\\
+&\|f_{j}\|^{1-p}_{L_{a}^{p}(\omega)}\int_{\mathbb{D}\backslash r_{3}\mathbb{D}}|f_{j}(\zeta)|^{p}\kappa(\zeta)\frac{(1-|\zeta|^{2})^{s+1}}{\sigma(\zeta)}
\int_{\mathbb{D}}\frac{|g(u)-g(\zeta)|}{|1-\bar{\zeta}u|^{t+3}}d\mu(u)dA(\zeta)\nonumber\\
\lesssim&\varepsilon\sup_{\zeta\in\mathbb{D}}\bigg(\frac{(1-|\zeta|^{2})^{s+1}}{\sigma(\zeta)}
\int_{\mathbb{D}}\frac{|g(u)-g(\zeta)|}{|1-\bar{\zeta}u|^{t+3}}d\mu(u)\bigg)\|f_{j}\|^{1-p}_{L_{a}^{p}(\omega)}\nonumber\\
+&\varepsilon\|g\|_{\mathcal{B}}\|f_{j}\|_{L_{a}^{p}(\omega)}
\lesssim\varepsilon\|g\|_{\mathcal{B}}.
\end{align*}
Therefore, we conclude that
\begin{equation}\label{RC14}
|I_{2,j}|\lesssim\varepsilon\|g\|_{\mathcal{B}},\quad \textmd{for}~j>k_{3}.
\end{equation}
It follows from \eqref{decomposition}, \eqref{RC10} and \eqref{RC14} that $T_{\mu,\omega}:L_{a}^{p}(\omega)\to L_{a}^{1}(\omega)$ is compact, which completes the whole proof.
\end{proof}

\section{Toeplitz operators with $L^1(\omega)$ symbols}
Recall that for a given interval $I\subseteq \mathbb{T}$, the \textit{Carleson square} is
\[
S(I)=\{z\in\mathbb{D}:\frac{z}{|z|}\in I,1-|I|\leq |z|<1\},
\]
where $|\cdot|$ is the normalized arc measure on $\mathbb{T}$. For each $a\in\mathbb{D}\backslash\{0\}$, denote the Carleson square $S(I_a)$ by $S(a)$, where
$I_{a}=\bigg\{z\in\mathbb{T}:|\arg (a\bar{z})|\leq\frac{1-|a|}{2}\bigg\}.$

Next, we turn to the boundedness of $T_{\varphi,\omega}$ on $L^1_a(\omega)$, and start with a weighted version of a result of Zhu \cite[Proposition 12]{zhu1989}, which will be used in the proof of Theorem \ref{T:bound10}.
\begin{proposition}\label{P:weightedzhu} Let $\omega$ be a $\mathcal{D}$ weight and $\varphi\in L^{1}(\omega)$. If there exists $r\in(0,1)$ such that
\[
\sup_{z\in\mathbb{D}}\frac{\ln\bigg(\frac{1}{1-|z|}\bigg)}{\omega(D(z,r))}\int_{D(z,r)}|\varphi(\zeta)|\omega(\zeta)dA(\zeta)
<+\infty,\]
then the Toeplitz operator $T_{\varphi,\omega}$ is bounded on $L_{a}^{1}(\omega)$.
\end{proposition}
\begin{proof}
Assume that $\omega\in\mathcal{D}$ and $\varphi \in L^{1}(\omega)$. Note that for any $z\in\mathbb{D}$,
\[
\omega(S(z))=\frac{(1-|z|)}{\pi}\int_{|z|}^{1}\omega(s)sds
\geq\frac{(1-|z|)}{2\pi}\int_{\frac{1+|z|}{2}}^{1}\omega(s)ds
\gtrsim(1-|z|)\widehat{\omega}(z).
\]
Then we have
\begin{equation}\label{E:S}
\omega(S(z))\asymp(1-|z|)\widehat{\omega}(z),\quad z\in\mathbb{D}.
\end{equation}
For any fixed $r\in(0,1)$, by \cite[pp. 65]{D1}, we note that $D(z,r)\subseteq S((1-h)e^{i\arg z})$ for $|z|>r$, where $h=\max\bigg\{1-\frac{|z|-r}{1-r|z|},2\sin^{-1}\big(\frac{r(1-|z|^{2})}{|z|(1-r^{2})}\big)\bigg\}$.
Recall a radial weight $\vartheta\in\widehat{\mathcal{D}}$ if and only if
there exists $\alpha>0$ ($\alpha$ only depends on $\vartheta$) such that
\[
\widehat{\vartheta}(r)\lesssim \bigg(\frac{1-r}{1-t}\bigg)^{\alpha}\widehat{\vartheta}(t),\quad 0\leq r\leq t<1
\]
(see \cite[Lemma A]{PRS}).
Together with
\eqref{E:S}, we deduce
\[
\omega(D(z,r))\leq \omega(S((1-h)e^{i\arg z}))\asymp h\widehat{\omega}(1-h)\lesssim\widehat{\omega}(z)(1-|z|), \quad |z|>r,
\]
which yields
\begin{equation}\label{new00}
\omega(D(z,r))\lesssim\widehat{\omega}(z)(1-|z|),\quad z\in\mathbb{D}.
\end{equation}
Hence, we obtain
\begin{eqnarray*}
\sup_{z\in\mathbb{D}}\frac{\ln\bigg(\frac{1}{1-|z|}\bigg)}{\widehat{\omega}(z)(1-|z|)}
\int_{D(z,r)}|\varphi(\zeta)|\omega(\zeta)dA(\zeta)<+\infty.
\end{eqnarray*}
Consequently,
\begin{eqnarray*}
\sup_{z\in\mathbb{D}}\frac{1}{\widehat{\omega}(z)(1-|z|)}
\int_{D(z,r)}\ln\bigg(\frac{1}{1-|\zeta|}\bigg)|\varphi(\zeta)|\omega(\zeta)dA(\zeta)<+\infty.
\end{eqnarray*}
By Lemma \ref{P:D}, we know that
\[
d\mu(\zeta):=\ln\bigg(\frac{1}{1-|\zeta|}\bigg)|\varphi(\zeta)|\omega(\zeta)dA(\zeta)
\]
is a Carleson measure for $L_{a}^{1}(\omega)$. By Proposition \ref{lla45}, there exists $r_{0}\in(0,1)$ such that for any $z\in\mathbb{D}$ with $r_{0}\leq|z|<1$,
\[
\|K_{z}^{\omega}\|_{L_{a}^{1}(\omega)}\asymp \ln\bigg(\frac{1}{1-|z|}\bigg).
\]
Let $\kappa(z)=(1-|z|)^{-1}\widehat{\omega}(z)$. Then for $f\in L_{a}^{1}(\omega )$, we get
\begin{align*}
\|T_{\varphi,\omega}(f)\|_{L_{a}^{1}(\omega)}
\leq&\int_{\mathbb{D}}\int_{\mathbb{D}}|f(z)\varphi(z)K_{\zeta}^{\omega}(z)|\omega(z)dA(z)\omega(\zeta)dA(\zeta)\nonumber\\
=&\int_{\mathbb{D}}|f(z)\varphi(z)|\int_{\mathbb{D}}|K_{z}^{\omega}(\zeta)|\omega(\zeta)dA(\zeta)\omega(z)dA(z)\nonumber\\
=&\int_{D(0,r_{0})}|f(z)\varphi(z)|\|K_{z}^{\omega}\|_{L_{a}^{1}(\omega)}\omega(z)dA(z)\nonumber\\
+&
\int_{\mathbb{D}\backslash D(0,r_{0})}|f(z)\varphi(z)|\|K_{z}^{\omega}\|_{L_{a}^{1}(\omega )}\omega(z)dA(z)\nonumber\\
\lesssim&\|\varphi\|_{L^{1}(\omega)}\cdot\sup_{z\in D(0,r_{0})}\|K_{z}^{\omega}\|_{L_{a}^{1}(\omega )}\cdot\sup_{z\in D(0,r_{0})}|f(z)|\nonumber\\
+&
\int_{\mathbb{D}\backslash D(0,r_{0})}|f(z)\varphi(z)|\ln\bigg(\frac{1}{1-|z|}\bigg)\omega(z)dA(z)\nonumber\\
\lesssim&\|\varphi\|_{L^{1}(\omega)}\cdot\sup_{z\in D(0,r_{0})}\|K_{z}^{\omega}\|_{L_{a}^{1}(\omega )}\cdot
\sup_{z\in D(0,r_{0})}\sup_{I:z\in S(I)}\frac{1}{\kappa(S(I))}\int_{S(I)}|f(\zeta)|\kappa(\zeta)dA(\zeta)\nonumber\\
+&
\int_{\mathbb{D}\backslash D(0,r_{0})}|f(z)\varphi(z)|\ln\bigg(\frac{1}{1-|z|}\bigg)\omega(z)dA(z)\nonumber\\
\lesssim&\|\varphi\|_{L^{1}(\omega)}\cdot\sup_{z\in D(0,r_{0})}\|K_{z}^{\omega}\|_{L_{a}^{1}(\omega)}\cdot
\frac{\|f\|_{L_{a}^{1}(\omega)}}{\widehat{\kappa}(r_{0})(1-r_{0})}+\int_{\mathbb{D}}|f(z)|d\mu(z)\nonumber\\
\lesssim&\|f\|_{L_{a}^{1}(\omega)},
\end{align*}
where the last third inequality is due to Lemma \ref{prop-reverse} and \cite[Lemma 2.5]{PR}, the last second inequality is due to Lemma \ref{prop-reverse} and \eqref{E:S}, and the last inequality is due to the fact that $\mu$ is a Carleson measure for $L_{a}^{1}(\omega)$.
Thus,
$T_{\varphi,\omega}:L_{a}^{1}(\omega)\rightarrow L_{a}^{1}(\omega)$ is bounded, as desired.
\end{proof}

Now we are ready to prove Theorem \ref{T:bound10}.
\begin{proof}[Proof of Theorem \ref{T:bound10}] From the assumption and \eqref{new00}, we obtain that there exist constants $\varepsilon>0$ and $r\in(0,1)$ such that
\begin{equation}\label{eq00000}
\sup_{z\in\mathbb{D}}\frac{1}{\widehat{\omega}(z)^{\varepsilon}(1-|z|)^{\varepsilon}\omega(D(z,r))}
\int_{D(z,r)}|\varphi(\zeta)|\omega(\zeta)dA(\zeta)<+\infty.
\end{equation}
Let
$$
\kappa(s)=\frac{\widehat{\omega}(s)}{1-s}, \quad s\in[0,1).
$$
From Lemma \ref{prop-reverse}, we know that $\kappa$ is regular. Then there exists a constant $C>1$ which only depends on $\omega$ such that
\[
C^{-1}\kappa(s)(1-s)\leq \widehat{\kappa}(s)\leq C\kappa(s)(1-s), \quad s\in[0,1).
\]
Hence,
\begin{equation}\label{eq000}
\kappa(s)\leq C\widehat{\kappa}(0)(1-s)^{\frac{1}{C}-1}
\end{equation}
for $s\in[0,1)$. Moreover, for any $z\in\mathbb{D}$ and $t\in(0,\infty)$, we have
\begin{equation}\label{eq0000}
\ln\bigg(\frac{1}{1-|z|}\bigg)\leq\frac{1}{t}(1-|z|)^{-t}.
\end{equation}
Let
\[
I=
\sup_{z\in\mathbb{D}}\frac{\ln\bigg(\frac{1}{1-|z|}\bigg)}{\omega(D(z,r))}\int_{D(z,r)}|\varphi(\zeta)|\omega(\zeta)dA(\zeta).
\]
It follows from (\ref{eq00000}), (\ref{eq000}) and (\ref{eq0000}) that there exists a positive constant $C'$ such that
\begin{align*}
I\leq&\frac{C}{\varepsilon(1+C)}\sup_{z\in\mathbb{D}}\frac{1}{(1-|z|)^{\varepsilon+\frac{\varepsilon}{C}}
\omega(D(z,r))}\int_{D(z,r)}|\varphi(\zeta)|\omega(\zeta)dA(\zeta)\nonumber\\\\
\leq&\frac{C^{\varepsilon+1}\widehat{\kappa}(0)^{\varepsilon}}{\varepsilon(1+C)}
\sup_{z\in\mathbb{D}}\frac{1}{\widehat{\omega}(z)^{\varepsilon}
(1-|z|)^{\varepsilon}\omega(D(z,r))}\int_{D(z,r)}|\varphi(\zeta)|\omega(\zeta)dA(\zeta)\nonumber\\\\
\leq&C'\sup_{z\in\mathbb{D}}\frac{1}{\widehat{\omega}(z)^{\varepsilon}(1-|z|)^{\varepsilon}\omega(D(z,r))}
\int_{D(z,r)}|\varphi(\zeta)|\omega(\zeta)dA(\zeta)\nonumber\\\\
<&+\infty.
\end{align*}
The conclusion then follows from Proposition \ref{P:weightedzhu}.
\end{proof}

\section*{Acknowledgements}
We thank Professor D. H. Luecking for his explanation of some details in \cite{Lu1}. This work is partially supported by National Natural Science Foundation of China. Y. Duan thanks School of Mathematics of Fudan University for the support of his visit to Shanghai. Z. Wang thanks School of Mathematics and Statistics, Northeast Normal University for the support of his visit to Changchun. We thank the anonymous referees for many valuable suggestions which make the paper more readable.


\end{document}